\documentclass[letterpaper]{amsart}
\usepackage[margin=1in]{geometry} 
\usepackage[utf8]{inputenc}
\usepackage{amsmath,amsthm,amssymb,amsfonts,mathrsfs}
\usepackage{dsfont}
\usepackage{bookmark}
\usepackage{xcolor}
\usepackage[hyperpageref]{backref}

\renewcommand*{\backref}[1]{}

\renewcommand*{\backrefalt}[4]{%
  \ifcase #1
    (No citations)
  \else
    \space #2%
  \fi
}

\newcommand{\FR}{\mathrm{FR}}

\numberwithin{equation}{section}
    
\newtheorem{theorem}{Theorem}[section]
\newtheorem{lemma}[theorem]{Lemma}
\newtheorem{proposition}[theorem]{Proposition}
\newtheorem{corollary}[theorem]{Corollary}

\theoremstyle{definition}
\newtheorem{definition}[theorem]{Definition}

\theoremstyle{remark}
\newtheorem{remark}[theorem]{Remark}

\title[Uncertainty principles]{Uncertainty Principles, Spectral Localization, and Singular Schrödinger Operators on Compact Manifolds}
\author[A. Iosevich]{Alex Iosevich}
\author[C. Park]{Chamsol Park}
\address{Department of Mathematics, University of Rochester, Rochester, NY 14627, USA}
\email{alex.iosevich@rochester.edu}
\address{Department of Mathematics, University of Rochester, Rochester, NY 14627, USA}
\email{cpark63@ur.rochester.edu}
\thanks{The first-listed author was supported by NSF DMS-2506858.}
\date{}

\usepackage{pgfplots}
\pgfplotsset{compat=1.16}

\begin{document}

\begin{abstract}
We establish uncertainty principles on compact Riemannian manifolds without boundary by combining restriction estimates for orthonormal systems with spectral projection bounds for Laplace-Beltrami and Schr\"odinger operators. Our results relate the size of the support of spectrally localized functions to the cardinality of the underlying spectral cluster and to Fourier-ratio type quantities. We obtain analogues for Schr\"odinger operators with singular potentials belonging to Kato and scaling-critical classes. As an application, we prove uniqueness results for recovery from incomplete spectral data on compact manifolds. Under curvature assumptions, including nonpositive and negative sectional curvatures, we also prove logarithmically improved uncertainty principles associated with shrinking spectral windows.
\end{abstract}

\maketitle

\section{Introduction}

Uncertainty principles play a central role in harmonic analysis, spectral theory, and signal recovery. Broadly speaking, such principles assert that a function and its frequency representation cannot both be sharply localized. In recent work \cite{IosevichMayeli2025ACHA}, the authors studied this philosophy in the setting of finite abelian groups, showing that restriction-theoretic methods lead to strengthened uncertainty principles and applications to sparse signal recovery from incomplete frequency data. The purpose of the present paper is to develop analogous uncertainty principles in the setting of compact Riemannian manifolds and Schr\"odinger operators with singular potentials, using spectral projection estimates and Fourier-ratio type quantities.

Let $\Delta_g$ be the Laplace-Beltrami operator on a compact Riemannian manifold $(M, g)$ without boundary, and the $e_j$ are the eigenfunctions of $-\Delta_g$ associated with the eigenvalues $\lambda_j$ such that
\begin{align*}
    -\Delta_g e_j=\lambda_j^2 e_j,\quad 0\leq \lambda_1\leq \lambda_2\leq \lambda_3 \leq \cdots \leq +\infty.
\end{align*}
We define
\begin{align}\label{setup:flq-FR-P-S}
    \begin{split}
        \|f\|_{\widehat{\ell}^q} &:= \left( \sum_j |\langle f,e_j\rangle_{L^2(M)}|^q \right)^{\frac1q}. \\
        \FR_{q}(f)&:=\frac{\|f\|_{\widehat{\ell}^q}}{\|f\|_{L^2 (M)}}, \\
        P&:=\sqrt{-\Delta_g}, \\
        S&:=\{j: \lambda_j\in [\lambda, \lambda+1) \}.
    \end{split}
\end{align}

The guiding principle behind this paper is that strong spectral localization should force quantitative delocalization in physical space. In Euclidean harmonic analysis, this phenomenon is reflected in uncertainty principles and restriction estimates. Our goal is to show that analogous effects persist on compact Riemannian manifolds and remain valid in the presence of singular Schr\"odinger potentials. The Fourier ratio $\FR_q(f)$ serves as a quantitative measure of spectral concentration, and the estimates proved below show that functions with strong spectral localization cannot concentrate on sets of arbitrarily small measure. The logarithmic improvements obtained under curvature assumptions reflect the sharper spectral cluster estimates available on manifolds with nonpositive or negative sectional curvature. These uncertainty principles also lead to uniqueness results for recovery from incomplete spectral data, extending the restriction-theoretic recovery framework developed in \cite{IosevichMayeli2025ACHA}.

Our first result establishes a restriction-theoretic uncertainty principle on compact manifolds. It shows that the size of the support of a spectrally localized function is controlled by both the dimension of the underlying spectral cluster and the Fourier ratio $\FR_q(f)$.

\begin{theorem}\label{thm:UP-FR-on}
    Let $(M, g)$ be a smooth compact Riemannian manifold without boundary with dimension $n\geq 2$. Suppose $\lambda\geq 1$, $\Pi_\lambda=\mathds{1}_{[\lambda, \lambda+1]} (\sqrt{-\Delta_g})$ is the spectral projection of $-\Delta_g$ onto the spectral cluster $\Pi_\lambda L^2 (M)$, and $(f_j)_{j\in J}\subset \Pi_\lambda L^2 (M)$ is any orthonormal system. Suppose $f$ is a nonzero function on $M$ such that $f$ is supported in $E\subset M$ and
    \begin{align*}
        f(y)=\sum_{j\in J} \langle f, f_j\rangle_{L^2 (M)} f_j (y).
    \end{align*}
    Then for $q\geq 2$ there is a uniform constant $C>0$ such that
    \begin{align}\label{ineq:UP-FR-on}
        (\# J)^{\frac{1}{2}} |E|^{\frac{q-2}{2q}}&\geq \frac{1}{C\lambda^{\sigma(q)}}\cdot \frac{\|f\|_{L^2 (M)}}{\left(\sum_{j\in J} |\langle f, f_j \rangle_{L^2 (M)}|^{2\alpha (q)} \right)^{\frac{1}{2\alpha(q)}} },
    \end{align}
    where
    \begin{align}\label{setup:sigma-alpha}
        \begin{split}
            \begin{cases}
                \sigma(q)=n\left(\frac{1}{2}-\frac{1}{q} \right)-\frac{1}{2}, \text{ and } \alpha(q)=\frac{q(n-1)}{2n}, & \text{if } \frac{2(n+1)}{n-1}\leq q\leq \infty, \\
                \sigma(q)=\frac{n-1}{2}\left(\frac{1}{2}-\frac{1}{q} \right), \text{ and } \alpha(q)=\frac{2q}{q+2}, & \text{if } 2\leq q\leq \frac{2(n+1)}{n-1}.
            \end{cases}
        \end{split}
    \end{align}
    Consequently, if we take $J=S=\{j: \lambda_j\in [\lambda, \lambda+1) \}$ and $f_j\equiv e_j$ for all $j\in S$, then
    \begin{align}\label{ineq:UP-FR-eig}
        (\# S)^{\frac{1}{2}} |E|^{\frac{q-2}{2q}}\geq \frac{1}{C\lambda^{\sigma(q)}}\cdot \frac{1}{\FR_{2\alpha(q)}(f)}, \quad q\geq 2.
    \end{align}
\end{theorem}

The proof of Theorem \ref{thm:UP-FR-on} relies on the restriction theorem for orthonormal systems established in \cite{FrankSabin2017Spectral}; see also \cite{FrankSabin2017RestrictionTheorem, RenZhang2024Improved}. This theorem provides the analytic mechanism that converts restriction estimates into quantitative uncertainty principles involving the Fourier ratio. We next establish analogues of \eqref{ineq:UP-FR-eig} for Schr\"odinger operators with singular potentials. We define the time independent Schr\"odinger operator $H_V$
\begin{align*}
    H_V=-\Delta_g+V,
\end{align*}
where $V$ is a real-valued singular potential. As usual, let $\mathcal{K}(M)$ be the set of all Kato class potentials $V$, where the potentials $V$ satisfy
\begin{align*}
    \lim_{\delta\to 0+} \sup_{x\in M} \int_{d_g (x, y)<\delta} |V(y)| W_n (d_g (x, y))\:dy=0,
\end{align*}
where
\begin{align}\label{setup-Wn}
    \begin{split}
        W_n (r)=\begin{cases}
            r^{2-n}, & \text{if } n\geq 3, \\
            \log (2+r^{-1}), & \text{if } n=2.
        \end{cases}
    \end{split}
\end{align}
Here, $d_g$ is the Riemannian distance and $dy$ is the Riemannian volume form on $(M, g)$. It is known that $L^s (M)\subset \mathcal{K}(M)\subset L^1 (M)$ for any $s>\frac{n}{2}$, and both $\mathcal{K}(M)$ and $L^{n/2}(M)$ have the same scaling properties and either one is not contained in the other. For background on Kato class potentials and their relationship with scaling-critical spaces, we refer the reader to \cite{Simon1982Semigroup}.

If $V\in \mathcal{K}(M)$, by \cite{BlairSireSogge2021Quasimode}, we know that $H_V$ is essentially self-adjoint and bounded from below. If $V\in L^{\frac{n}{2}}(M)$, by \cite{BlairHuangSireSogge2022UniformSobolev}, again, $H_V$ is essentially self-adjoint, bounded from below. Adding a uniform positive constant, we may assume that the $H_V$ is positive. Since $M$ is compact, the spectrum of $H_V$ is discrete as it is for $V\equiv 0$, i.e., $-\Delta_g$. If $V\in \mathcal{K}(M)$, it is also known that the associated eigenfunctions are continuous (see, e.g., \cite{LiYau1986Parabolic}, \cite{Sturm1993SchrodingerSemigroups}, \cite{Guneysu2012OnGeneralizedSchrodingerSemigroups}). We can then write
\begin{align*}
    H_V e_j^V=\tau_j^2 e_j^V,
\end{align*}
where the $e_j^V$ are the $L^2$-normalized eigenfunctions of $H_V$ associated with the eigenvalues $\tau_j^2$ with $0\leq \tau_1\leq \tau_2\leq \cdots$ and $\{\tau_j\}$ is discrete. It is known that $\{e_j^V\}$ is also an orthonormal basis on $L^2 (M)$. If $V\equiv 0$, $H_V=-\Delta_g$. Let $E_j^V$ be the projection operator associated with the eigenfunction $e_j^V$ and eigenvalue $\tau_j$, which is analogous to $E_j$ for the $V\equiv 0$ case, i.e., we write
\begin{align*}
    \mathds{1}_{[\lambda, \lambda+1]}(P)f=\sum_{\lambda_j\in [\lambda, \lambda+1]} E_jf,\quad \mathds{1}_{[\lambda, \lambda+1]}(\sqrt{H_V})f=\sum_{\tau_j\in [\lambda, \lambda+1]} E_j^Vf.
\end{align*}

The passage from the free Laplace-Beltrami operator to Schr\"odinger operators with singular potentials is highly nontrivial. In the scaling-critical regime $V\in L^{n/2}(M)$, the potential may exhibit singularities comparable to those arising in critical elliptic theory, while Kato class potentials allow even rougher local behavior. Nevertheless, recent spectral projection estimates for singular Schr\"odinger operators show that many of the harmonic analytic properties of the free Laplace-Beltrami operator persist in this setting. Our goal is to show that the corresponding uncertainty principles remain stable under these singular perturbations.

Using the spectral projection estimates in \cite{BlairSireSogge2021Quasimode} and \cite{BlairHuangSireSogge2022UniformSobolev}, we obtain the following analogues of \eqref{ineq:UP-FR-eig} for singular Schr\"odinger operators.

\begin{theorem}\label{thm:hv-Kato-or-Ln2}
    Let $(M, g)$ be a smooth compact Riemannian manifold without boundary with dimension $n\geq 2$. Let $V\in \mathcal{K}(M)\cup L^{\frac{n}{2}}(M)$. Suppose $\lambda\geq 1$, and $f$ is a nonzero function on $M$ such that
    \begin{align*}
        f=\mathds{1}_{[\lambda, \lambda+1]} (\sqrt{H_V})f,\quad \mathrm{supp}(f)\subset E\subset M.
    \end{align*}
    Let $S_1=\{j: \tau_j\in [\lambda, \lambda+1]\}$.
    \begin{enumerate}
        \item If $V\in \mathcal{K}(M)\cap L^{\frac{n}{2}}(M)$, then there is a uniform constant $C_V>0$ such that
        \begin{align}\label{ineq:UP-hv-Kato-and-Ln2}
            (\# S_1)^{\frac{1}{2}}|E|^{\frac{q-2}{2q}}\geq \frac{1}{C_V\lambda^{\sigma(q)}},\quad 2\leq q\leq \infty.
        \end{align}
        \item If $V\in L^{\frac{n}{2}}(M)$, then there is a uniform constant $C_V>0$ such that
        \begin{align}\label{ineq:UP-hv-Ln2}
            (\# S_1)^{\frac{1}{2}}|E|^{\frac{q-2}{2q}}\geq \frac{1}{C_V\lambda^{\sigma(q)}},\quad \text{when } \begin{cases}
                2<q\leq \frac{2n}{n-4}, & \text{if } n\geq 5, \\
                2<q<\infty, & \text{if } n\in \{3, 4\}.
            \end{cases}
        \end{align}
        \item There is a uniform constant $C_V>0$ such that
        \begin{align}\label{ineq:UP-hv-mixed}
            (\# S_1)|E|^{\frac{2}{n}}\geq \frac{1}{C_V \lambda}, \quad \text{where } \begin{cases}
                n\geq 3 \text{ and } V\in L^{\frac{n}{2}}(M), \text{ or} \\
                n=2 \text{ and } V\in \mathcal{K}(M).
            \end{cases}
        \end{align}
    \end{enumerate}
\end{theorem}

On manifolds with nonpositive or negative sectional curvature, spectral cluster estimates admit logarithmic improvements over the general case. Such gains originate in the improved geometry of the geodesic flow and the resulting improvements in wave propagation estimates. A natural question is whether these logarithmic spectral improvements also lead to stronger uncertainty principles. Our next result shows that this is indeed the case, both for the free Laplace-Beltrami operator and, conditionally, for singular Schr\"odinger operators.

Our next result gives logarithmically improved analogues of Theorem \ref{thm:hv-Kato-or-Ln2}.

\begin{theorem}\label{thm:UP-nonposit-curv}
    Let $(M, g)$ be a smooth compact Riemannian manifold without boundary with dimension $n\geq 2$ with nonpositive sectional curvatures. Suppose, for $V\in L^{\frac{n}{2}}(M)$,
    \begin{align*}
        f=\mathds{1}_{[\lambda, \lambda+(\log \lambda)^{-1}]} (\sqrt{H_V})f, \quad \lambda \geq 1,
    \end{align*}
    and $f$ is supported in $E\subset M$. Depending on the curvatures of $M$, we define $\delta_n=\delta_n(q)$ as
    \begin{align}\label{setup:delta-n}
        \begin{split}
            \delta_n (q)=\begin{cases}
            \sigma(q), & \text{if $M$ has nonpositive sectional curvatures everywhere and $2<q\leq \frac{2(n+1)}{n-1}$}, \\
            \frac{1}{2}, & \text{if $M$ has negative sectional curvatures everywhere and $2<q\leq \frac{2(n+1)}{n-1}$}, \\
            \frac{1}{2}, & \text{if $M$ has nonpositive sectional curvatures everywhere and $\frac{2(n+1)}{n-1}<q\leq \infty$}.
            \end{cases}
        \end{split}
    \end{align}
    and let
    \begin{align*}
        S_2^V=\{j: \tau_j \in [\lambda, \lambda+(\log \lambda)^{-1}]\}.
    \end{align*}
    \begin{enumerate}
        \item Let $V\equiv 0$. Suppose $M$ has either nonpositive sectional curvatures or negative sectional curvatures. Then there is a uniform constant $C>0$ such that
        \begin{align}\label{ineq:V=0-nonpositive}
            (\# S_2)^{\frac{1}{2}}|E|^{\frac{q-2}{2q}}\geq \frac{(\log \lambda)^{\delta_n (q)}}{C\lambda^{\sigma (q)}},\quad 2<q\leq \infty,
        \end{align}
        where $S_2=S_2^V$ when $V\equiv 0$.
        \item Let $V\in L^{\frac{n}{2}}(M)$.
            \begin{enumerate}
                \item If $M$ has nonpositive sectional curvatures, then there is a uniform constant $C_V>0$ such that
                \begin{align}\label{ineq:s2-sup-crit}
                    (\# S_2^V)^{\frac{1}{2}} |E|^{\frac{q-2}{2q}}\geq \frac{(\log (2+\lambda))^{\delta_n (q)}}{C_V \lambda^{\sigma(q)}}, \quad \text{when } \begin{cases}
                        \frac{2(n+1)}{n-1}<q\leq \frac{2n}{n-4}, & \text{for } n\geq 5, \\
                        \frac{2(n+1)}{n-1}<q<\infty, & \text{for } n\in \{3, 4\}.
                    \end{cases}
                \end{align}
                \item If $M$ has either nonpositive sectional curvatures or negative sectional curvatures, then there is a uniform constant $C_V>0$ such that, for $\delta_n$ as in \eqref{setup:delta-n},
                \begin{align}\label{ineq:s2-crit}
                    (\# S_2^V)^{\frac{1}{2}} |E|^{\frac{q-2}{2q}}\geq \frac{(\log (2+\lambda))^{\delta_n (q)}}{C_V \lambda^{\sigma(q)}}, \quad q=\frac{2(n+1)}{n-1}.
                \end{align}
            \end{enumerate}
        \item  Let $V\in \mathcal{K}(M) \cap L^{\frac{n}{2}}(M)$. Suppose $M$ has nonpositive sectional curvatures. Let $\epsilon'>0$ and $r>\frac{n}{2}$. For $\delta_n (q)$ as in \eqref{setup:delta-n}, we define $\kappa(q)$ by
        \begin{align}\label{setup-kappa}
            \kappa(q)=\begin{cases}
                \frac{1}{2}, & \text{if } n=2, 2<q\leq 6, V\in \mathcal{K}(M), \text{ and } M \text{ has negative curvatures}, \\
                \frac{1}{2}, & \text{if } n=2, 6<q\leq \infty, V\in \mathcal{K}(M), \text{ and } M \text{ has nonpositive curvatures}, \\
                \frac{1}{2}, & \text{if } n\geq 3, q=\frac{2(n+1)}{n-1}, V\in \mathcal{K}(M)\cap L^{\frac{n}{2}} (M), \text{ and } M \text{ has negative curvatures}, \\
                \frac{1}{2}-\epsilon', & \text{if } n\in \{3, 4\}, \frac{2n}{n-1}<q<\frac{2(n+1)}{n-1}, V\in \mathcal{K}(M)\cap L^{\frac{n}{2}} (M), \text{ and } M \text{ has negative curvatures}, \\
                \frac{1}{2}, & \text{if } n\in \{3, 4\}, \frac{2(n+1)}{n-1}<q<\infty, V\in \mathcal{K}(M)\cap L^{\frac{n}{2}}(M), \text{ and $M$ has nonpositive curvatures}, \\
                \frac{1}{2}-\epsilon', & \text{if } n\geq 5, \frac{2n}{n-1}<q<\frac{2(n+1)}{n-1}, V\in \mathcal{K}(M)\cap L^{\frac{n}{2}} (M), \text{ and $M$ has negative curvatures}, \\
                \frac{1}{2}, & \text{if } n\geq 5, \frac{2(n+1)}{n-1}<q\leq \frac{2n}{n-4}, V\in \mathcal{K}(M)\cap L^{\frac{n}{2}}(M), \text{ and $M$ has nonpositive curvatures}, \\
                \delta_n(q), & \text{if } n\geq 2, 2<q\leq \infty, \text{ and } V\in L^r (M).
            \end{cases}
        \end{align}
        Then there is a uniform constant $C_V>0$ such that we have
            \begin{align}\label{ineq:s2-all}
                (\# S_2^V)^{\frac{1}{2}} |E|^{\frac{q-2}{2q}}\geq \frac{(\log (2+\lambda))^{\kappa(q)}}{C_V \lambda^{\sigma(q)}}.
            \end{align}
    \end{enumerate}
\end{theorem}

We note that, compared to recent work for uncertainty principles on compact Riemannian manifolds (see, e.g., \cite{IosevichMayeliWyman2024Uncertainty, IosevichMayeliWyman2026Spectral, IosevichPark2026UP-singular-potentials}, and references therein), Theorem \ref{thm:UP-FR-on}-\ref{thm:UP-nonposit-curv} do not assume homogeneity on manifolds and do not contain an assumption on a ``tubular neighborhood of $E$''. The results in \cite{IosevichMayeliWyman2024Uncertainty, IosevichMayeliWyman2026Spectral, IosevichPark2026UP-singular-potentials} exploit additional geometric structure, such as homogeneity assumptions or concentration along tubular neighborhoods. By contrast, Theorem \ref{thm:UP-FR-on}--Theorem \ref{thm:UP-nonposit-curv} require neither hypothesis. In this sense, the present paper provides a complementary restriction-theoretic approach to uncertainty principles on compact manifolds.

As we said, the estimates of the form
\begin{align}\label{est:bss-bhss}
    \|\mathds{1}_{[\lambda, \lambda+\epsilon(\lambda)]} (\sqrt{H_V})\|_{L^q (M)\to L^2 (M)},\quad \lambda\geq 1,\quad \lambda^{-1}\leq \epsilon(\lambda)\leq 1
\end{align}
have been studied in \cite{BlairSireSogge2021Quasimode} and \cite{BlairHuangSireSogge2022UniformSobolev}. Specifically, the case $\epsilon(\lambda)\equiv 1$ is considered in \cite{BlairSireSogge2021Quasimode} for $V\in \mathcal{K}(M)\cap L^{\frac{n}{2}}(M)$, and the various cases of $\lambda^{-1}\leq \epsilon(\lambda)\leq 1$ for $V\in L^{\frac{n}{2}} (M)$ are considered in \cite{BlairHuangSireSogge2022UniformSobolev}. The estimates \eqref{est:bss-bhss} have been studied by using the uniform Sobolev estimates. For the proofs of \eqref{ineq:s2-sup-crit} and \eqref{ineq:s2-crit}, we shall use the estimates considered in \cite{BlairSireSogge2021Quasimode} and \cite{BlairHuangSireSogge2022UniformSobolev}.

For the proof of \eqref{ineq:s2-all}, we briefly review submanifold analogues of \eqref{est:bss-bhss}, which are of the form
\begin{align}\label{est:hv-restriction}
    \|\mathds{1}_{[\lambda, \lambda+\epsilon(\lambda)]} (\sqrt{H_V})\|_{L^q (\Sigma)\to L^2 (M)},\quad \lambda\geq 1,\quad \lambda^{-1}\leq \epsilon(\lambda)\leq 1,
\end{align}
where $\Sigma$ is a submanifold of $M$. These estimates \eqref{est:hv-restriction} are considered in \cite{BlairPark2025LqEstimates} for $V\in L^{\frac{n}{2}}(M)$ by using the uniform Sobolev (trace) estimates (see also \cite{BlairPark2026Resolvent} as related work), and considered in \cite{HuangWangZhang2026restriction} for $V\in \mathcal{K}(M) \cap L^{\frac{n}{2}} (M)$ by the perturbation of the wave kernel using Duhamel's principle. Using a perturbative comparison argument based on Duhamel's principle, together with the logarithmic spectral cluster estimates for the free Laplace-Beltrami operator obtained in \cite{HassellTacy2015improvement, HuangSogge2025Curvature}, one can transfer the logarithmic improvement from $\sqrt{-\Delta_g}$ to $\sqrt{H_V}$ when $V\in \mathcal{K}(M)\cap L^{\frac{n}{2}}(M)$.

\begin{proposition}\label{prop:hv-Kato-and-Ln2}
     Let $(M, g)$ be a smooth compact Riemannian manifold without boundary with dimension $n\geq 2$. Suppose $(M, g)$ has nonpositive sectional curvatures everywhere. Let $V\in \mathcal{K}(M) \cap L^{\frac{n}{2}} (M)$. Then there is a uniform constant $C_V>0$ such that
    \begin{align*}
        \|\mathds{1}_{[\lambda, \lambda+(\log \lambda)^{-1}]} (\sqrt{H_V}) f\|_{L^q (M)}\leq C_V\frac{\lambda^{\sigma (q)}}{(\log \lambda)^{\kappa(q)}} \|f\|_{L^2 (M)},
    \end{align*}
    where $\kappa(q)$ is as in \eqref{setup-kappa}.
\end{proposition}

We shall use Proposition \ref{prop:hv-Kato-and-Ln2} to prove \eqref{ineq:s2-all}. Proposition \ref{prop:hv-Kato-and-Ln2} is of independent interest, since it transfers logarithmic spectral cluster improvements from the Laplace--Beltrami operator to singular Schr\"odinger operators with potentials in $\mathcal K(M)\cap L^{n/2}(M)$. For completeness, we provide a proof of Proposition \ref{prop:hv-Kato-and-Ln2} in \S \ref{sec:Pf-Prop-hv-Kato-and-Ln2} by combining perturbative arguments based on Duhamel's principle in \cite{HuangWangZhang2026restriction} with the logarithmic spectral cluster estimates in \cite{HuangSogge2025Curvature, HassellTacy2015improvement}.

\begin{remark}[Fourier ratio] The significance of Theorem \ref{thm:UP-FR-on} is that it strengthens the basic uncertainty principle whenever the Fourier ratio is small. The quantity $\FR_{2\alpha(q)}(f)$ measures spectral concentration, and stronger concentration leads to a larger lower bound in \eqref{ineq:UP-FR-eig}.
    We note that, for $f=\mathds{1}_{[\lambda, \lambda+1]}(P)$,
    \begin{align}\label{ineq:FR-2-alpha}
        \FR_{2\alpha(q)}(f)\leq 1.
    \end{align}
    Indeed, by a standard property of $\ell^p$ spaces, we know $\ell^p \subset \ell^q$ when $0<p\leq q\leq \infty$. We also know that, by the definition of $\alpha(q)$ in \eqref{setup:sigma-alpha}, we have $\alpha(q)\geq 1$ for $2\leq q\leq \infty$. It follows that
    \begin{align*}
        \FR_{2\alpha(q)} (f)=\frac{\left(\sum_{j\in J} |\langle f, f_j \rangle_{L^2 (M)}|^{2\alpha (q)} \right)^{\frac{1}{2\alpha(q)}} }{\|f\|_{L^2 (M)}}\leq 1.
    \end{align*}
    This and \eqref{ineq:UP-FR-eig} imply
    \begin{align}\label{ineq:FR-better}
        (\# S)^{\frac{1}{2}} |E|^{\frac{q-2}{2q}}\geq \frac{1}{C\lambda^{\sigma(q)}}\cdot \frac{1}{\FR_{2\alpha(q)} (f)}\geq \frac{1}{C\lambda^{\sigma(q)}},\quad 2\leq q\leq \infty.
    \end{align}
    The lower bound can be larger if $\FR_{2\alpha(q)}(f)$ is of the form $B(\lambda)$, where $B(\lambda)$ is monotonically decreasing in $\lambda$, e.g., $B(\lambda)=\lambda^{-\alpha'}$ for some $\alpha'>0$, or $B(\lambda)=(\log \lambda)^{-\beta'}$ for some $\beta'>0$, etc.
\end{remark}

We shall prove Theorem \ref{thm:UP-FR-on} in \S \ref{sec:pf-of-thm:UP-FR-on}, Theorem \ref{thm:hv-Kato-or-Ln2} in \S \ref{sec:pf-of-thm:hv-Kato-or-Ln2}, Theorem \ref{thm:UP-nonposit-curv} in \S \ref{sec:pf-of-thm:UP-nonposit-curv}, and Proposition \ref{prop:hv-Kato-and-Ln2} in \S \ref{sec:Pf-Prop-hv-Kato-and-Ln2}. In \S \ref{sec:recovery}, we illustrate signal recovery on manifolds in the sense of \cite{DonohoStark1989} by using uncertainty principles in this paper. For notation, we write $A\lesssim B$ if $A\leq CB$ for a uniform constant $C$ depending on $M$, $n=\dim M$, or other fixed parameters, but independent of the frequency $\lambda\geq 1$. The notation $A\approx B$ means $A\lesssim B$ and $B\lesssim A$.

\section{Proof of Theorem \ref{thm:UP-FR-on}}\label{sec:pf-of-thm:UP-FR-on}

The argument follows the strategy of the proof of \cite[Theorem 3.5]{IosevichMayeli2025ACHA}, adapted to the setting of spectral clusters on compact manifolds. To make use of the proof of \cite[Theorem 3.5]{IosevichMayeli2025ACHA}, we shall use the following theorem \cite[Theorem 2]{FrankSabin2017Spectral}, which can be thought of as a manifold analogue of the restriction theorems discussed in \cite{IosevichMayeli2025ACHA}.

\begin{theorem}[\cite{FrankSabin2017Spectral}]\label{FS:thm2}
    Let $(M, g)$ be a smooth compact Riemannian manifold of dimension $n\geq 2$ without boundary. For any $\lambda\geq 1$, let $\Pi_\lambda=\mathds{1}_{[\lambda, \lambda+1)} (\sqrt{-\Delta_g})$ be the spectral projection of $-\Delta_g$ onto the spectral cluster $\Pi_\lambda L^2 (M)$. Then there is a $C>0$ such that for any orthonormal system $(f_j)_{j\in J}\subset \Pi_\lambda L^2 (M)$, for any $(\nu_j)_{j\in J}\subset \mathbb{C}$ and any $\lambda\geq 1$, we have
    \begin{align*}
        \left\|\sum_{j\in J} \nu_j |f_j|^2 \right\|_{L^{q/2}(M)}\leq C\lambda^{2\sigma(p)} \left(\sum_{j\in J} |\nu_j|^{\alpha (q)} \right)^{\frac{1}{\alpha (q)}}.
    \end{align*}
\end{theorem}

We now prove Theorem \ref{thm:UP-FR-on}. Since $f=\sum_{j\in J} \langle f, f_j \rangle_{L^2 (M)} f_j$, by H\"older's inequality,
\begin{align*}
    |f(y)|&\leq \sum_{j\in J} 1\cdot |\langle f, f_j \rangle_{L^2 (M)} f_j(y)|\\
    &\leq \left(\sum_{j\in J} 1^2 \right)^{\frac{1}{2}}\left(\sum_{j \in J} |\langle f, f_j \rangle_{L^2 (M)}|^2 |f_j (y)|^2 \right)^{\frac{1}{2}}=(\# J)^{\frac{1}{2}} \left(\sum_{j \in J} |\langle f, f_j \rangle_{L^2 (M)}|^2 |f_j (y)|^2 \right)^{\frac{1}{2}}.
\end{align*}
By this, H\"older's inequality, and Theorem \ref{FS:thm2}, since $f$ is supported in $E\subset M$, if $q>2$, then
\begin{align*}
    \|f\|_{L^2 (M)}^2=\int_E |f(y)|^2\:dy&\leq (\# J) \int_E 1\cdot \sum_{j \in J} |\langle f, f_j \rangle_{L^2 (M)}|^2 |f_j (y)|^2\:dy \\
    &\leq (\# J) \left(\int_E 1^{\frac{q}{q-2}}\:dy \right)^{\frac{q-2}{q}} \left(\int_E \left|\sum_{j \in J} |\langle f, f_j \rangle_{L^2 (M)}|^2 |f_j (y)|^2 \right|^{\frac{q}{2}}\:dy \right)^{\frac{2}{q}}\\
    &=(\# J) |E|^{\frac{q-2}{q}} \left\| \sum_{j \in J} |\langle f, f_j \rangle_{L^2 (M)}|^2 |f_j|^2 \right\|_{L^{\frac{q}{2}}(E)} \\
    &\leq (\# J) |E|^{\frac{q-2}{q}} \left\| \sum_{j \in J} |\langle f, f_j \rangle_{L^2 (M)}|^2 |f_j|^2 \right\|_{L^{\frac{q}{2}}(M)} \\
    &\lesssim (\# J)|E|^{\frac{q-2}{q}} \lambda^{2\sigma(q)} \left(\sum_{j\in J} |\langle f, f_j\rangle_{L^2(M)}|^{2\alpha (q)} \right)^{\frac{1}{\alpha(q)}},
\end{align*}
and thus,
\begin{align*}
    \|f\|_{L^2 (M)}\lesssim (\# J)^{\frac{1}{2}} |E|^{\frac{q-2}{2q}} \lambda^{\sigma(q)} \left(\sum_{j\in J} |\langle f, f_j\rangle_{L^2(M)}|^{2\alpha (q)} \right)^{\frac{1}{2\alpha(q)}}.
\end{align*}
Hence,
\begin{align*}
    (\# J)^{\frac{1}{2}} |E|^{\frac{q-2}{2q}}&\gtrsim \frac{1}{\lambda^{\sigma(q)}}\cdot \frac{\|f\|_{L^2 (M)}}{\left(\sum_{j\in J} |\langle f, f_j \rangle_{L^2 (M)}|^{2\alpha (q)} \right)^{\frac{1}{2\alpha(q)}} },
\end{align*}
proving \eqref{ineq:UP-FR-on}. The endpoint case $q=2$ follows by the same argument, noting that the factor involving $|E|$ disappears.

\section{Proof of Theorem \ref{thm:hv-Kato-or-Ln2}}\label{sec:pf-of-thm:hv-Kato-or-Ln2}

\subsection{\texorpdfstring{$V\in \mathcal{K}(M) \cap L^{\frac{n}{2}}(M)$}{V in K(M) and L(n/2)(M)} cases}
We prove \eqref{ineq:UP-hv-Kato-and-Ln2} in this subsection. The proof of \eqref{ineq:UP-hv-Kato-and-Ln2} follows the same general strategy as the proof of Theorem \ref{thm:UP-FR-on}, combined with spectral projection estimates for singular Schr\"odinger operators. Since $f=\mathds{1}_{[\lambda, \lambda+1]}(\sqrt{H_V})f$, we can write
\begin{align*}
    f(y)=\sum_{\tau_j\in [\lambda, \lambda+1]} \langle f, e_j^V\rangle_{L^2(M)}e_j^V(y).
\end{align*}
Since $S_1=\{j: \tau_j\in [\lambda, \lambda+1]\}$, we can write, by H\"older's inequality,
\begin{align}\label{ineq:|f|-3.1}
    |f(y)|\leq (\# S_1)^{\frac{1}{2}}\left(\sum_{\tau_j\in [\lambda, \lambda+1]} |\langle f, e_j^V\rangle_{L^2 (M)}|^2 |e_j^V(y)|^2 \right)^{\frac{1}{2}}.
\end{align}
Recall that by \cite[Corollary 1.4]{BlairSireSogge2021Quasimode}
\begin{align}\label{ineq:bss-cor.1.4}
    \left\|\sum_{\tau_j\in [\lambda, \lambda+1]} E_j^V f \right\|_{L^q (M)}\lesssim \lambda^{\sigma(q)} \|f\|_{L^2 (M)},\quad 2\leq q\leq \infty,\quad V\in \mathcal{K}(M)\cap L^{\frac{n}{2}}(M).
\end{align}
By \eqref{ineq:|f|-3.1}-\eqref{ineq:bss-cor.1.4}, H\"older's inequality, triangle inequality (i.e., Minkowski's inequality), and orthogonality, since $f$ is supported in $E\subset M$, if $q>2$ (the case $q=2$ is trivial as above), we have
\begin{align*}
    \|f\|_{L^2 (M)}^2&\leq (\# S_1)\int_E 1\cdot \sum_{\tau_j\in [\lambda, \lambda+1]} |\langle f, e_j^V \rangle_{L^2 (M)}|^2 |e_j^V (y)|^2\:dy \\
    &\leq (\# S_1) |E|^{\frac{q-2}{q}} \left\| \sum_{\tau_j \in [\lambda, \lambda+1]} |\langle f, e_j^V \rangle_{L^2 (M)}|^2 |e_j^V|^2 \right\|_{L^{\frac{q}{2}}(M)} \\
    &\leq (\# S_1)|E|^{\frac{q-2}{q}} \sum_{\tau_j\in [\lambda, \lambda+1]} \left\| |\langle f, e_j^V\rangle_{L^2 (M)}|^2 \|e_j^V\|^2 \right\|_{L^{\frac{q}{2}}(M)} \\
    &=(\# S_1)|E|^{\frac{q-2}{q}} \sum_{\tau_j\in [\lambda, \lambda+1]} |\langle f, e_j^V\rangle_{L^2 (M)}|^2 \|e_j^V\|_{L^q (M)}^2 \\
    &\lesssim (\# S_1)|E|^{\frac{q-2}{q}}\cdot \lambda^{2\sigma(q)} \sum_{\tau_j\in [\lambda, \lambda+1]} |\langle f, e_j^V\rangle_{L^2 (M)}|^2 \\
    &=(\# S_1)|E|^{\frac{q-2}{q}}\cdot \lambda^{2\sigma(q)}\|f\|_{L^2 (M)}^2.
\end{align*}
Dividing both sides by $\|f\|_{L^2 (M)}^2$, we obtain \eqref{ineq:UP-hv-Kato-and-Ln2}.

\subsection{\texorpdfstring{$V\in L^{\frac{n}{2}}(M)$ and $n\geq 3$}{V in L(n/2)(M) cases}}
We show \eqref{ineq:UP-hv-Ln2} in this subsection. Recall that by \cite[(1.10)]{BlairHuangSireSogge2022UniformSobolev}
\begin{align}\label{ineq:bhss-(1.10)}
    \left\|\sum_{\tau_k\in [\lambda, \lambda+1]} E_k^V f \right\|_{L^q (M)}\lesssim \lambda^{\sigma(q)} \|f\|_{L^2 (M)},\quad V\in L^{\frac{n}{2}}(M), \quad \text{where } \begin{cases}
        2<q\leq \frac{2n}{n-4}, & \text{if } n\geq 5, \\
        2<q<\infty, & \text{if } n\in \{3, 4\}.
    \end{cases}
\end{align}
Replacing \eqref{ineq:bss-cor.1.4} by \eqref{ineq:bhss-(1.10)} in the proof of \eqref{ineq:UP-hv-Kato-and-Ln2} immediately yields \eqref{ineq:UP-hv-Ln2}.

\subsection{\texorpdfstring{$V\in \mathcal{K}(M)$ and $n=2$, or $V\in L^{\frac{n}{2}}(M)$ and $n\geq 3$}{Mixed cases}}
We prove \eqref{ineq:UP-hv-mixed} in this subsection. The argument combines ideas from the proofs of \cite[Theorem 3.5]{IosevichMayeli2025ACHA} and \cite[Theorem 3.6]{IosevichMayeli2025ACHA}, adapted to the spectral decomposition associated with $H_V$. By hypothesis, we write
\begin{align*}
    f=\mathds{1}_{[\lambda, \lambda+1]}(\sqrt{H_V})f=\sum_{\tau_j\in [\lambda, \lambda+1]} E_j^Vf.
\end{align*}
We focus on the case $n\geq 3$, since the two-dimensional argument is entirely analogous. By H\"older's inequality,
\begin{align*}
    |f(y)|&\leq \sum_{\tau_j \in [\lambda, \lambda+1]}1\cdot |E_j^V f(y)| \\
    &\leq \left(\sum_{j\in S_1} 1^{\frac{2n}{n+2}} \right)^{\frac{n+2}{2n}} \left(\sum_{\tau_j \in [\lambda, \lambda+1]} |E_j^V f(y)|^{\frac{2n}{n-2}} \right)^{\frac{n-2}{2n}}=(\# S_1)^{\frac{n+2}{2n}} \left(\sum_{\tau_j \in [\lambda, \lambda+1]} |E_j^V f(y)|^{\frac{2n}{n-2}} \right)^{\frac{n-2}{2n}},
\end{align*}
and so,
\begin{align*}
    |f(y)|^{\frac{2n}{n+2}}\leq (\# S_1)\left(\sum_{\tau_j\in [\lambda, \lambda+1]} |E_j^V f(y)|^{\frac{2n}{n-2}} \right)^{\frac{n-2}{n+2}},\quad y\in M.
\end{align*}
Since $f$ is supported in $E\subset M$, integrating both sides over $M$, we have, by H\"older's inequality,
\begin{align*}
    \|f\|_{L^{\frac{2n}{n+2}}(M)}^{\frac{2n}{n+2}}=\|f\|_{L^{\frac{2n}{n+2}}(E)}^{\frac{2n}{n+2}}&\leq (\# S_1)\int_E 1\cdot \left(\sum_{\tau_j \in [\lambda, \lambda+1]} |E_j^V f(y)|^{\frac{2n}{n-2}} \right)^{\frac{n-2}{n+2}}\:dy \\
    &\leq (\# S_1)\left(\int_E 1^{\frac{n+2}{4}}\:dy \right)^{\frac{4}{n+2}}\left(\int_E \sum_{\tau_j\in [\lambda, \lambda+1]} |E_j^V f(y)|^{\frac{2n}{n-2}}\:dy \right)^{\frac{n-2}{n+2}},
\end{align*}
and so,
\begin{align}\label{ineq:f-2n/(n+2)(M)}
    \begin{split}
        \|f\|_{L^{\frac{2n}{n+2}}(M)}&\leq (\# S_1)^{\frac{n+2}{2n}}\left(|E|^{\frac{4}{n+2}} \right)^{\frac{n+2}{2n}} \left(\int_E \sum_{\tau_j\in [\lambda, \lambda+1]} |E_j^V f(y)|^{\frac{2n}{n-2}}\:dy \right)^{\frac{n-2}{2n}} \\
        &=(\# S_1)^{\frac{n+2}{2n}} |E|^{\frac{2}{n}} \left(\sum_{\tau_j \in [\lambda, \lambda+1]} \int_E |E_j^V f(y)|^{\frac{2n}{n-2}}\:dy \right)^{\frac{n-2}{2n}} \\
        &=(\# S_1)^{\frac{n+2}{2n}} |E|^{\frac{2}{n}} \left(\sum_{\tau_j \in [\lambda, \lambda+1]} \|E_j^V f\|_{L^{\frac{2n}{n-2}}(E)}^{\frac{2n}{n-2}} \right)^{\frac{n-2}{2n}} \\
        &\leq (\# S_1)^{\frac{n+2}{2n}} |E|^{\frac{2}{n}} \left(\sum_{\tau_j \in [\lambda, \lambda+1]} \|E_j^V f\|_{L^{\frac{2n}{n-2}}(M)}^{\frac{2n}{n-2}} \right)^{\frac{n-2}{2n}}.
    \end{split}
\end{align}
On the other hand, replacing $f$ by $E_j^Vf$ in \eqref{ineq:bhss-(1.10)} for $n\geq 3$ (if $n=2$, we use \eqref{ineq:bss-cor.1.4} with $q=\infty$), we know
\begin{align}
    \|E_j^V f\|_{L^{\frac{2n}{n-2}}(M)}\lesssim \lambda^{\sigma\left(\frac{2n}{n-2}\right)}\|E_j^V f\|_{L^2 (M)}.
\end{align}
Also, we know, by orthogonality, for $\tau_j\in [\lambda, \lambda+1]$,
\begin{align}\label{ineq:EjV-L2(M)}
    \|E_j^V f\|_{L^2 (M)}^2\leq \sum_{\tau_k\in [\lambda, \lambda+1]} \|E_k^V f\|_{L^2 (M)}^2=\|f\|_{L^2(M)}^2,\quad \text{i.e., } \|E_j^V f\|_{L^2 (M)}\leq \|f\|_{L^2 (M)}.
\end{align}
Combining these \eqref{ineq:f-2n/(n+2)(M)}-\eqref{ineq:EjV-L2(M)} together yields
\begin{align}\label{ineq:L-2n/(n+2)norm-f}
    \begin{split}
        \|f\|_{L^{\frac{2n}{n+2}}(M)}&\lesssim (\# S_1)^{\frac{n+2}{2n}} |E|^{\frac{2}{n}}\cdot \lambda^{\sigma\left(\frac{2n}{n-2} \right)}\left(\sum_{\tau_j\in [\lambda, \lambda+1]}1 \right)^{\frac{n-2}{2n}}\cdot \|f\|_{L^2 (M)} \\
        &=(\# S_1)|E|^{\frac{2}{n}}\cdot \lambda^{\sigma\left(\frac{2n}{n-2} \right)} \|f\|_{L^2 (M)}.
    \end{split}
\end{align}
By duality and \eqref{ineq:bhss-(1.10)}, we know that
\begin{align*}
    \|f\|_{L^2 (M)}\lesssim \lambda^{\sigma\left(\frac{2n}{n-2} \right)}\|f\|_{L^{\frac{2n}{n+2}}(M)},\quad \text{for } f=\mathds{1}_{[\lambda, \lambda+1]}(\sqrt{H_V})f.
\end{align*}
By this and \eqref{ineq:L-2n/(n+2)norm-f}, we have
\begin{align*}
    \|f\|_{L^{\frac{2n}{n+2}}(M)}\lesssim (\# S_1)|E|^{\frac{2}{n}}\cdot \lambda^{2\sigma\left(\frac{2n}{n-2} \right)} \|f\|_{L^{\frac{2n}{n+2}} (M)}=(\# S_1)|E|^{\frac{2}{n}}\cdot \lambda^1 \|f\|_{L^{\frac{2n}{n+2}} (M)}.
\end{align*}
Dividing both sides by $\|f\|_{L^{\frac{2n}{n+2}}(M)}$, we obtain \eqref{ineq:UP-hv-mixed}.

\section{Proof of Theorem \ref{thm:UP-nonposit-curv}}\label{sec:pf-of-thm:UP-nonposit-curv}

\begin{figure}[h]
\centering

\begin{tikzpicture}[scale=1]

% Main spectral axis
\draw[thick,->] (0,0) -- (12,0);
\node at (12.4,0) {$\lambda$};

% Standard window
\draw[fill=gray!25] (2,-0.25) rectangle (4,0.25);
\node at (3,0.7) {unit window};
\node at (3,-0.7) {$[\lambda,\lambda+1]$};

% Arrow
\draw[->,thick] (4.5,0) -- (6,0);

% Logarithmic window
\draw[fill=gray!55] (7,-0.25) rectangle (7.8,0.25);
\node at (7.4,0.7) {shrinking window};
\node at (7.4,-0.9)
{$[\lambda,\lambda+(\log\lambda)^{-1}]$};

% Improvement arrow
\draw[->,thick] (7.4,-1.5) -- (7.4,-2.6);

% Gain text
\node at (7.4,-3.2)
{$(\log\lambda)^{\delta_n(q)}$ gain};

% Curvature note
\node at (7.4,-4.1)
{nonpositive / negative curvature};

\end{tikzpicture}

\caption{Logarithmically shrinking spectral windows yield improved uncertainty principles on manifolds with nonpositive or negative sectional curvature.}
\label{fig:logwindow}
\end{figure}
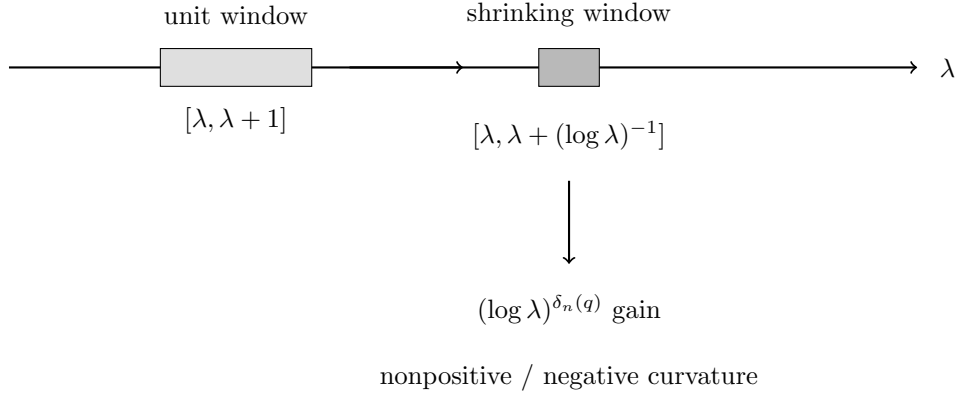

\subsection{Vanishing potential cases}
Let $V\equiv 0$. Combining the results in \cite{HuangSogge2025Curvature} (see also \cite{BlairSogge2019logarithmic}, \cite{BlairHuangSogge2022Improved}, \cite{CanzaniGalkowski2023APDE} and see the references therein, as related work) and \cite{HassellTacy2015improvement}, we know
\begin{align}\label{ineq:hs-ht-main}
    \|\mathds{1}_{[\lambda, \lambda+(\log \lambda)^{-1}]}(P)f\|_{L^{\frac{2(n+1)}{n-1}}(M)} \lesssim \frac{\lambda^{\frac{n-1}{2(n+1)}}}{(\log (2+\lambda))^{\delta_n}} \|f\|_{L^2 (M)}.
\end{align}
Combining \eqref{ineq:hs-ht-main} with the argument used in the proof of \eqref{ineq:UP-hv-Kato-and-Ln2} yields \eqref{ineq:V=0-nonpositive}.

\subsection{\texorpdfstring{$V\in L^{\frac{n}{2}}(M)$ and $n\geq 3$ for super-critical exponents}{Super-critical exponent cases}}
Let $V\in L^{\frac{n}{2}}(M)$. The estimate \eqref{ineq:s2-sup-crit} follows by combining the proof of \eqref{ineq:UP-hv-Kato-and-Ln2} with the logarithmically improved spectral projection bounds
\begin{align*}
    \|\mathds{1}_{[\lambda, \lambda+(\log \lambda)^{-1}]}(\sqrt{H_V})f\|_{L^q (M)}
    \lesssim
    \frac{\lambda^{\sigma(q)}}{(\log (2+\lambda))^{1/2}}
    \|f\|_{L^2(M)},
\end{align*}
valid when
\begin{align*}
\begin{cases}
\frac{2(n+1)}{n-1}<q\leq \frac{2n}{n-4}, & \text{if } n\geq 5, \\
\frac{2(n+1)}{n-1}<q<\infty, & \text{if } n=3,4,
\end{cases}
\end{align*}
which follow from \cite[(1.20)]{BlairHuangSireSogge2022UniformSobolev}.

\subsection{\texorpdfstring{$V\in L^{\frac{n}{2}}(M)$ and $n\geq 3$ for critical exponents}{Critical exponent cases}}
Let $V\in L^{\frac{n}{2}}(M)$. Using \eqref{ineq:hs-ht-main} and the arguments in \cite[Section 3]{BlairHuangSireSogge2022UniformSobolev}, we have
\begin{align*}
   \|\mathds{1}_{[\lambda, \lambda+(\log \lambda)^{-1}]}(\sqrt{H_V})f\|_{L^q (M)} \lesssim \frac{\lambda^{\sigma(q)}}{(\log (2+\lambda))^{\delta_n}}\|f\|_{L^2 (M)},\quad \text{when } q=\frac{2(n+1)}{n-1}.
\end{align*}
Combining this estimate with the argument used in the proof of \eqref{ineq:UP-hv-Kato-and-Ln2} yields \eqref{ineq:s2-crit}.

\subsection{\texorpdfstring{$V\in \mathcal{K}(M)\cap L^{\frac{n}{2}}(M)$ or $V\in L^r(M)$}{Critically or less singular potential cases}}
Let $V\in \mathcal{K}(M) \cap L^{\frac{n}{2}}(M)$ or $V\in L^r$ with $r>\frac{n}{2}$. Combining Proposition \ref{prop:hv-Kato-and-Ln2} with the argument used in the proof of \eqref{ineq:UP-hv-Kato-and-Ln2} yields \eqref{ineq:s2-all}.

\section{Recovery from incomplete spectral data}\label{sec:recovery}

One of the classical applications of uncertainty principles is to the problem of recovery
from incomplete frequency data. In Euclidean and finite settings, this circle of ideas goes
back to Donoho and Stark \cite{DonohoStark1989}, who showed that sufficiently sparse
signals can be recovered uniquely even when some Fourier coefficients are missing. The
underlying mechanism is simple and robust: if two signals agree on the observed frequency
region, then their difference is simultaneously concentrated in physical space and supported
in the missing frequency region. An uncertainty principle then forces the difference to vanish
identically.

In the finite abelian group setting, this philosophy was developed further in
\cite{IosevichMayeli2025ACHA}, where restriction estimates were used to improve the
classical uncertainty principle and consequently enlarge the range of admissible missing
frequencies. The purpose of this section is to show that the uncertainty principles proved
above yield analogous recovery statements in the spectral setting on compact manifolds.

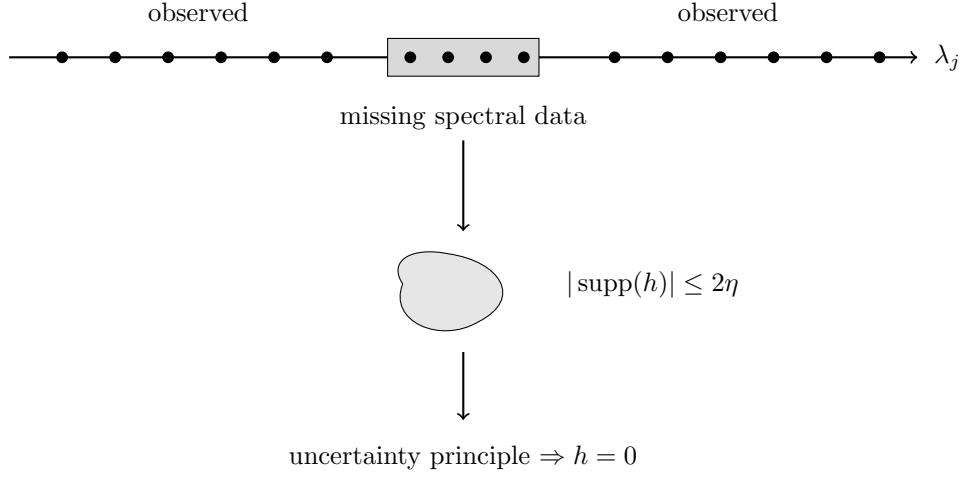
\begin{figure}[h]
\centering

\begin{tikzpicture}[scale=1]

% Spectral axis
\draw[thick,->] (0,0) -- (12,0);
\node at (12.4,0) {$\lambda_j$};

% Observed frequencies
\foreach \x in {0.7,1.4,2.1,2.8,3.5,4.2,8.0,8.7,9.4,10.1,10.8,11.5}
{
\filldraw[black] (\x,0) circle (2pt);
}

% Missing cluster
\draw[fill=gray!30] (5,-0.25) rectangle (7,0.25);

% Dots inside cluster
\foreach \x in {5.3,5.8,6.3,6.8}
{
\filldraw[black] (\x,0) circle (2pt);
}

% Labels
\node at (6,-0.8) {missing spectral data};
\node at (2.5,0.6) {observed};
\node at (9.5,0.6) {observed};

% Arrow downward
\draw[->,thick] (6,-1.1) -- (6,-2.3);

% Spatial support blob
\filldraw[fill=gray!20] (5.2,-3)
.. controls (5.0,-3.5) and (5.7,-3.8) ..
(6.2,-3.5)
.. controls (6.8,-3.2) and (6.5,-2.7) ..
(5.8,-2.6)
.. controls (5.3,-2.5) and (5.0,-2.7) ..
(5.2,-3);

\node at (8.5,-3) {$|\operatorname{supp}(h)|\leq 2\eta$};

% Final implication
\draw[->,thick] (6,-3.9) -- (6,-4.8);

\node at (6,-5.3)
{uncertainty principle $\Rightarrow h=0$};

\end{tikzpicture}

\caption{Recovery from incomplete spectral data. Agreement outside a missing spectral cluster forces the difference function to be spectrally localized. The uncertainty principle then prevents simultaneous spectral and spatial concentration.}
\label{fig:recovery}
\end{figure}

The role of the missing frequency set is now played by a spectral cluster. Let
\[
P=\sqrt{-\Delta_g},
\]
and denote by
\[
\Pi_I=\mathds{1}_I(P)
\]
the spectral projection associated with an interval
\[
I\subset [0,\infty).
\]
If
\[
f=\sum_j \langle f,e_j\rangle e_j,
\]
we interpret the coefficients
\[
\langle f,e_j\rangle,
\qquad
\lambda_j\notin I,
\]
as the observed spectral data, while the coefficients corresponding to
\[
\lambda_j\in I
\]
are regarded as missing.

The recovery mechanism is based on the following elementary observation.

\begin{lemma}[Donoho--Stark mechanism in the spectral setting]
\label{lem:DS-manifold}
Let $f,g\in L^2(M)$ and let $I\subset [0,\infty)$.
Suppose that
\[
\langle f,e_j\rangle
=
\langle g,e_j\rangle
\qquad
\text{whenever }
\lambda_j\notin I.
\]
Then
\[
h:=f-g
\]
satisfies
\[
h=\Pi_I h.
\]
In particular, $h$ is spectrally localized in the interval $I$.
\end{lemma}

\begin{proof}
We have
\[
h=\sum_j \langle h,e_j\rangle e_j,
\]
where
\[
\langle h,e_j\rangle
=
\langle f,e_j\rangle
-
\langle g,e_j\rangle.
\]
By assumption,
\[
\langle h,e_j\rangle=0
\qquad
\text{whenever }
\lambda_j\notin I.
\]
Hence
\[
h
=
\sum_{\lambda_j\in I}
\langle h,e_j\rangle e_j
=
\Pi_I h,
\]
as claimed.
\end{proof}

The uncertainty principles proved in the previous sections imply that a function with
sufficiently small support cannot be spectrally concentrated inside a narrow spectral window.
Combining this observation with Lemma \ref{lem:DS-manifold} yields uniqueness of recovery.

We begin with the free Laplace--Beltrami case.

\begin{definition}
Let $\eta>0$ and let $I\subset [0,\infty)$.
We say that a function
\[
f\in L^2(M)
\]
is uniquely recoverable among functions supported on sets of measure at most $\eta$
from spectral data outside $I$ if the following holds:

whenever
\[
g\in L^2(M)
\]
satisfies
\[
|\operatorname{supp}(f)|\leq \eta,
\qquad
|\operatorname{supp}(g)|\leq \eta,
\]
and
\[
\langle f,e_j\rangle
=
\langle g,e_j\rangle
\qquad
\text{for every }
\lambda_j\notin I,
\]
then
\[
f=g.
\]
\end{definition}

We now apply \eqref{ineq:UP-FR-eig}.

\begin{theorem}[Recovery from a missing spectral cluster]
\label{thm:recovery-main}
Let $(M,g)$ be a smooth compact Riemannian manifold without boundary of dimension
$n\geq 2$.
Let
\[
I_\lambda=[\lambda,\lambda+1],
\qquad
\lambda\geq 1,
\]
and define
\[
\mathcal C_\lambda
=
\{j:\lambda_j\in I_\lambda\}.
\]

Suppose that
\[
2<q\leq \infty.
\]
Then there exists a constant $C>0$ such that if
\begin{equation}\label{eq:recovery-condition-main}
(\# \mathcal C_\lambda)^{\frac12}
(2\eta)^{\frac{q-2}{2q}}
<
\frac{1}{C\lambda^{\sigma(q)}},
\end{equation}
every function
\[
f\in L^2(M)
\]
with
\[
|\operatorname{supp}(f)|\leq \eta
\]
is uniquely recoverable among functions supported on sets of measure at most $\eta$
from spectral data outside $I_\lambda$.
\end{theorem}

\begin{proof}
Suppose that
\[
f,g\in L^2(M)
\]
satisfy
\[
|\operatorname{supp}(f)|\leq \eta,
\qquad
|\operatorname{supp}(g)|\leq \eta,
\]
and
\[
\langle f,e_j\rangle
=
\langle g,e_j\rangle
\qquad
\text{whenever }
\lambda_j\notin I_\lambda.
\]

Define
\[
h=f-g.
\]
By Lemma \ref{lem:DS-manifold},
\[
h=\Pi_{I_\lambda}h.
\]
Moreover,
\[
\operatorname{supp}(h)
\subset
\operatorname{supp}(f)\cup \operatorname{supp}(g),
\]
so
\[
|\operatorname{supp}(h)|
\leq 2\eta.
\]

Suppose that
\[
h\neq 0.
\]

Since
\[
h=\Pi_{I_\lambda}h,
\]
the function $h$ is spectrally supported in the cluster
\[
I_\lambda=[\lambda,\lambda+1].
\]
Therefore \eqref{ineq:UP-FR-eig} applies and yields
\begin{equation}\label{ineq:spec-result-recovery}
(\# \mathcal C_\lambda)^{\frac12}
|\operatorname{supp}(h)|^{\frac{q-2}{2q}}
\geq
\frac{1}{C\lambda^{\sigma(q)}}
\frac{1}{\FR_{2\alpha(q)}(h)}.
\end{equation}

Since
\[
2\alpha(q)\geq 2,
\]
by \eqref{ineq:FR-2-alpha}, we have
\[
\FR_{2\alpha(q)}(h)
\leq
1.
\]
Hence
\[
(\# \mathcal C_\lambda)^{\frac12}
|\operatorname{supp}(h)|^{\frac{q-2}{2q}}
\geq
\frac{1}{C\lambda^{\sigma(q)}}.
\]

Using
\[
|\operatorname{supp}(h)|\leq 2\eta,
\]
we obtain
\[
(\# \mathcal C_\lambda)^{\frac12}
(2\eta)^{\frac{q-2}{2q}}
\geq
\frac{1}{C\lambda^{\sigma(q)}},
\]
contradicting \eqref{eq:recovery-condition-main} if the constant $C>0$ in \eqref{eq:recovery-condition-main} is chosen to be the same constant $C>0$ in \eqref{ineq:spec-result-recovery}. Therefore
\[
h=0,
\]
which implies that
\[
f=g.
\]
The proof is complete.
\end{proof}

\begin{remark}
The factor $2\eta$ appearing in
\eqref{eq:recovery-condition-main}
is the direct analogue of the factor $2|E|$ in the classical Donoho--Stark uniqueness
argument \cite{DonohoStark1989}. It arises because the difference of two functions
supported on sets of measure at most $\eta$ is supported on the union of the two supports.
\end{remark}

We next consider the Schr\"odinger setting.
Let
\[
P_V=\sqrt{H_V},
\qquad
\Pi_I^V=\mathds{1}_I(P_V),
\]
and let
\[
\tau_j
\]
denote the eigenvalues of $P_V$.

\begin{theorem}[Recovery for singular Schr\"odinger operators]
\label{thm:recovery-schrodinger}
Let $(M,g)$ be a smooth compact Riemannian manifold without boundary of dimension
$n\geq 2$, and let
\[
V\in \mathcal K(M)\cup L^{n/2}(M).
\]

Let
\[
I_\lambda=[\lambda,\lambda+1],
\]
and define
\[
\mathcal C_\lambda^V
=
\{j:\tau_j\in I_\lambda\}.
\]

Suppose that the exponent $q$ belongs to one of the admissible ranges in
Theorem \ref{thm:hv-Kato-or-Ln2}.
Then there exists a constant $C_V>0$ such that if
\begin{equation}\label{eq:recovery-schrodinger}
(\# \mathcal C_\lambda^V)^{\frac12}
(2\eta)^{\frac{q-2}{2q}}
<
\frac{1}{C_V\lambda^{\sigma(q)}},
\end{equation}
every function
\[
f\in L^2(M)
\]
with
\[
|\operatorname{supp}(f)|\leq \eta
\]
is uniquely recoverable among functions supported on sets of measure at most $\eta$
from its $H_V$-spectral data outside $I_\lambda$.
\end{theorem}

\begin{proof}
The proof is identical to the proof of
Theorem \ref{thm:recovery-main},
with Theorem \ref{thm:hv-Kato-or-Ln2}
replacing \eqref{ineq:UP-FR-eig}.
\end{proof}

The logarithmically improved uncertainty principles proved in the previous section lead
directly to stronger recovery statements on manifolds with nonpositive or negative curvature.

\begin{corollary}[Logarithmically improved recovery]
\label{cor:log-recovery}
Assume the hypotheses of
Theorem \ref{thm:UP-nonposit-curv}.
Let
\[
I_\lambda^{\log}
=
[\lambda,\lambda+(\log \lambda)^{-1}],
\]
and define
\[
\mathcal C_{\lambda,\log}
=
\{j:\lambda_j\in I_\lambda^{\log}\}.
\]

Suppose that
\begin{equation}\label{eq:log-recovery}
(\# \mathcal C_{\lambda,\log})^{\frac12}
(2\eta)^{\frac{q-2}{2q}}
<
\frac{(\log \lambda)^{\delta_n(q)}}
{C\lambda^{\sigma(q)}}.
\end{equation}

Then every function
\[
f\in L^2(M)
\]
with
\[
|\operatorname{supp}(f)|\leq \eta
\]
is uniquely recoverable among functions supported on sets of measure at most $\eta$
from spectral data outside
$I_\lambda^{\log}$.
\end{corollary}

\begin{proof}
If two admissible functions agree outside
$I_\lambda^{\log}$,
their difference is spectrally localized in
$I_\lambda^{\log}$.
Applying the logarithmically improved uncertainty principle from
Theorem \ref{thm:UP-nonposit-curv}
yields a contradiction unless the difference vanishes identically.
\end{proof}

\begin{remark}
The results in this section are uniqueness statements rather than algorithmic recovery
procedures. This distinction is deliberate. In finite abelian groups, uncertainty principles
can often be combined with explicit reconstruction algorithms, including the direct rounding
algorithm for binary signals developed in
\cite{IosevichMayeli2025ACHA}.
In the present manifold setting, the natural analogue of the missing frequency set is a
spectral cluster, and the uncertainty principles proved above show that functions supported
on sufficiently small sets are uniquely determined by incomplete spectral data whenever the
missing spectral window is sufficiently small.

It would be interesting to investigate quantitative stability estimates and explicit
reconstruction procedures in this geometric setting. These questions lie beyond the scope of
the present paper.
\end{remark}

\section{Proof of Proposition \ref{prop:hv-Kato-and-Ln2}} \label{sec:Pf-Prop-hv-Kato-and-Ln2}
In this section, we prove Proposition \ref{prop:hv-Kato-and-Ln2}. Throughout this section, we write
\begin{align*}
    & P=\sqrt{-\Delta_g},\quad P_V=\sqrt{H_V}, \\
    & \epsilon=\epsilon(\lambda)=(\log \lambda)^{-1}, \\
    & \mathds{1}_\lambda=\mathds{1}_{[\lambda-\epsilon, \lambda+\epsilon]},\quad \mathds{1}_{\lambda, l}(s)=\mathds{1}_{|s-\lambda|\in (2^l, 2^{l+1}]}(s),\quad \mathds{1}_{\leq 2\lambda}=\mathds{1}_{(-\infty, 2\lambda]}.
\end{align*}
Also, throughout this section, for simplicity, let us focus on the cases where
\begin{itemize}
    \item $M$ has negative sectional curvatures when $2<q\leq \frac{2(n+1)}{n-1}$, and
    \item $M$ has nonpositive sectional curvatures when $\frac{2(n+1)}{n-1}<q<\infty$
\end{itemize}
so that the log-improvements are all $\frac{1}{2}$, i.e., $(\log \lambda)^{\frac{1}{2}}$-improvements, whenever we apply spectral cluster bounds from \cite{BlairHuangSireSogge2022UniformSobolev, HuangSogge2025Curvature, HassellTacy2015improvement}, i.e., \eqref{log-window-est} below. The other cases follow similarly.

By \cite{BlairSireSogge2021Quasimode}, Cauchy-Schwarz inequality, and orthogonality, we know that
\begin{align}\label{large-window-est}
    \|\mathds{1}_{[\lambda, \lambda+\rho]} (P_V)\|_{L^2 (M)\to L^q (M)}\lesssim \lambda^{\sigma(q)}\rho^{\frac{1}{2}},\quad \text{for all } \rho\geq 1.
\end{align}
Similarly, if $M$ has either nonpositive and negative sectional curvatures with the exponent $q$ considered in this section, then by the work of \cite{HassellTacy2015improvement, HuangSogge2025Curvature, BlairSogge2019logarithmic, BlairHuangSogge2022Improved, BlairSireSogge2021Quasimode, BlairHuangSireSogge2022UniformSobolev}, we have, for $\epsilon=\epsilon(\lambda)=(\log \lambda)^{-1}$,
\begin{align}\label{log-window-est}
    \begin{split}
        & \|\mathds{1}_{[\lambda, \lambda+\epsilon]}(P)\|_{L^2 (M)\to L^q (M)}\lesssim \lambda^{\sigma(q)}\epsilon^{\frac{1}{2}} \quad \text{when } 2<q\leq \infty, \text{ and}\\
        & \|\mathds{1}_{[\lambda, \lambda+\epsilon]}(P_V)\|_{L^2 (M)\to L^q (M)}\lesssim \lambda^{\sigma(q)}\epsilon^{\frac{1}{2}} \quad \text{when } V\in L^{\frac{n}{2}}(M) \text{ and } \begin{cases}
        \frac{2(n+1)}{n-1}<q\leq \frac{2n}{n-4}, & \text{if } n\geq 5, \\
        \frac{2(n+1)}{n-1}<q<\infty, & \text{if } n=3, 4.
    \end{cases} 
    \end{split}
\end{align}
We shall use these estimates repeatedly throughout this section, often without further comments, especially in the case where $V\in \mathcal{K}(M)\cap L^{n/2}(M)\subset L^{n/2}(M)$ and $q=\frac{2n}{n-2}$.

Fix a nonnegative even function $\chi\in C_0^\infty (\mathbb{R})$ such that
\begin{align*}
    \mathrm{supp}(\chi)\subset (-\epsilon, \epsilon),\quad \chi\equiv 1 \text{ on } (-\epsilon/2, \epsilon/2).
\end{align*}
We also introduce a notation
\begin{align*}
    \chi_\lambda (s)=\chi (\lambda-s),\quad \text{when } \lambda\gg 1.
\end{align*}
By \cite{HuangSogge2025Curvature} and \cite{HassellTacy2015improvement}, we know that
\begin{align}\label{eqn:prop-(3)}
    \|\chi_\lambda (P)\|_{L^2 (M)\to L^q (M)}\lesssim \frac{\lambda^{\sigma (q)}}{(\log \lambda)^{\kappa(q)}},
\end{align}
where $\kappa(q)$ is as in \eqref{setup-kappa} (as we said, we set $k(q)=\frac{1}{2}$ for simplicity in what follows). For $s\geq 0$, we have $\chi (\lambda+s)=0$, and so,
\begin{align*}
    \chi(\lambda-s)=\chi(\lambda-s)+\chi(\lambda+s)=\frac{1}{\pi}\int_{-\infty}^\infty \hat{\chi}(t)e^{it\lambda} \cos (ts)\:dt.
\end{align*}
As in \cite{HuangSogge2021Weyl, HuangZhang2022PtwiseWeyl, HuangZhang2023SharpPtwiseWeyl, HuangWangZhang2026restriction}, by Duhamel's principle and spectral theorem, we can write
\begin{align*}
    (\cos tP_V)(x, y)-(\cos tP)(x, y)=\sum_{\lambda_j} \sum_{\tau_k} \int_M \left[-\int_0^t \frac{\sin (t-s)\lambda_j}{\lambda_j} \cos (s\tau_k)\:ds\right] e_j (x) e_j (z) e_k^V (z) e_k^V (y) V(z)\:dx.
\end{align*}
Here, using the spectral theorem, we choose $\{e_j\}$ and $\{e_k^V\}$ to be the bases with real-valued functions for convenience so that we do not need to consider the conjugates of $e_j$ and $e_k^V$ in what follows. By \cite[Lemma 2.3]{HuangSogge2021Weyl},
\begin{align}\label{eqn:prop-1.5-(4.1)}
    \begin{split}
        -\int_0^t \frac{\sin (t-s)\lambda_j}{\lambda_j} \cos (s\tau_k)\:ds=\begin{cases}
            \frac{\cos (t\lambda_j)-\cos (t\tau_k)}{\lambda_j^2-\tau^2}, & \text{if } \lambda_j\not=\tau_k, \\
            -\frac{t \sin (t\tau_k)}{2\tau_k}, & \text{if } \lambda_j=\tau_k.
        \end{cases}
    \end{split}
\end{align}
If we set
\begin{align*}
    m(s, \tau)=\frac{\cos (ts)-\cos (t\tau)}{s^2-\tau^2},\quad \text{when } s\not=\tau,
\end{align*}
then
\begin{align*}
    \lim_{s\to \tau} m(s, \tau)=-\frac{t\sin (ts)}{2s},
\end{align*}
and so, in \eqref{eqn:prop-1.5-(4.1)}, for simplicity, we focus only on the case where
\begin{align*}
    -\int_0^t \frac{\sin (t-s)\lambda_j}{\lambda_j} \cos (s\tau_k)\:ds=m(\lambda_j, \tau_k) \quad \text{and} \quad \lambda_j\not=\tau_k.
\end{align*}
We can then write the difference of the kernels as
\begin{align*}
    \chi_\lambda (P_V)(x, y)-\chi_\lambda (P)(x, y)&=\sum_{\lambda_j}\sum_{\tau_k} \frac{\chi_\lambda (\lambda_j)-\chi_\lambda (\tau_k)}{\lambda_j^2-\tau^2} e_j (x) e_j (z) e_k^V (z) e_k^V (y) V(z)\:dz \\
    &=\sum_{\lambda_j}\sum_{\tau_k} \int_M m(\lambda_j, \tau_k) e_j (x) e_j (z) e_k^V (z) e_k^V (y) V(z)\:dz \\
    &=:K(x, y).
\end{align*}
If $K$ is the operator whose integral kernel is $K(x, y)$, then we want to show that the operator norm $\|K\|_{L^2 (M)\to L^q (M)}$ satisfies the desired bounds in Proposition \ref{prop:hv-Kato-and-Ln2}. In what follows, for notational convenience, if we consider the kernel, say, $\tilde{K}(x, y)$, then $\tilde{K}$ denotes the operator whose kernel is $\tilde{K}(x, y)$, and vice versa. As in \cite{HuangWangZhang2026restriction}, we consider five different cases.
\begin{itemize}
    \item Case 1: $|\tau_k-\lambda|\leq \epsilon$ and $|\lambda_j-\lambda|\leq \epsilon$.
    \item Case 2: $|\tau_k-\lambda|\leq \epsilon$, $|\lambda_j-\lambda|\in (2^l, 2^{l+1}]$, and $\epsilon\leq 2^l\leq \lambda$.
    \item Case 3: $|\lambda_j-\lambda|\leq \epsilon$, $|\tau_k-\lambda|\in (2^l, 2^{l+1}]$, and $\epsilon\leq 2^l\leq \lambda$.
    \item Case 4: $|\lambda_j-\lambda|\leq \epsilon$, and $\tau_k\geq 2\lambda$.
    \item Case 5: $|\tau_k-\lambda|\leq \epsilon$, and $\lambda_j\geq 2\lambda$.
\end{itemize}
We shall consider Case 1, 3, 4, 5, and 2 in order. Case 1, 3, 4 are almost the same as in \cite[Case 1, 3, 4]{HuangWangZhang2026restriction}. Case 5 and 2 are slightly changed here, but in fact simpler in Case 5, since we are not going to consider the bootstrap argument of \cite[Case 5]{HuangWangZhang2026restriction}. In Case 2, we shall use scaling arguments to obtain some ``uniform Sobolev''-type estimates in the sense of \cite{KenigRuizSogge1987UniformSobolev, S.Huang-Sogge2014Resolvent, DSFKenigSalo2014Forum, BourgainShaoSoggeYao2015Resolvent, Hickman2020UniformResolventEstimates, BlairSireSogge2021Quasimode, BlairHuangSireSogge2022UniformSobolev} to consider the exponents $q$ less than the critical exponent $\frac{2(n+1)}{n-1}$, say, $q>\frac{2n}{n-1}$ for the higher $n$ (see also \cite{BlairPark2025LqEstimates, BlairPark2026Resolvent} and \cite[\S 4, \S 7]{HuangWangZhang2026restriction} for related work), and this will be a difference. If one is familiar with the computations in \cite{HuangWangZhang2026restriction}, one can directly jump into Case 5 and 2, although we shall consider all the cases in this section for the sake of completeness.

\subsection{\texorpdfstring{Case 1: $|\tau_k-\lambda|\leq \epsilon, |\lambda_j-\lambda|\leq \epsilon$}{Case 1}}
In this case, we have
\begin{align}\label{eqn:prop-(4)}
    |m(\lambda_j, s)|+\epsilon|\partial_s m(\lambda_j, s)|\lesssim (\lambda \epsilon(\lambda))^{-1},\quad \text{for } |s-\lambda|\leq \epsilon(\lambda).
\end{align}
We write
\begin{align*}
    \sum_{|\lambda_j-\lambda|\leq \epsilon} \sum_{|\tau_k-\lambda|\leq \epsilon} \int_M m(\lambda_j, \tau_k) e_j (x) e_j (z) e_k^V (z) e_k^V (y) V(z)\:dz=K_1 (x, y)+K_2 (x, y),
\end{align*}
where
\begin{align*}
    & K_1 (x, y)=\sum_{|\lambda_j-\lambda|\leq \epsilon} \sum_{|\tau_k-\lambda|\leq \epsilon} \int_M \int_{\lambda-\epsilon}^{\lambda+\epsilon} \partial_s m (\lambda_j, s) \mathds{1}_{[\lambda-\epsilon, \tau_k]} (s) e_j (x) e_j (z) e_k^V (z) e_k^V (y)V(z)\:dz\:ds, \\
    & K_2 (x, y)=\sum_{|\lambda_j-\lambda|\leq \epsilon} \int_M \sum_{|\tau_k-\lambda|\leq \epsilon} m(\lambda_j, \lambda-\epsilon) e_j (x) e_j (z) e_k^V (z) e_k^V (y)V(z)\:dz\:ds.
\end{align*}
We compute $K_2 (x, y)$ only, since the computation for $K_1$ is similar. We note that
\begin{align*}
    K_2 f(x)&=\int_M K_2 (x, y)f(y)\:dy \\
    &=\sum_{|\lambda_j-\lambda|\leq \epsilon} m(\lambda_j, \lambda-\epsilon) e_j (x) \int_M e_j (z) V(z)\sum_{|\tau_k-\lambda|\leq \epsilon} e_k^V (z) \left[\int_M f(y) e_k^V (y)\:dy\right]dz \\
    &=\sum_{|\lambda_j-\lambda|\leq \epsilon} m (\lambda_j, \lambda-\epsilon) e_j (x)\int_M e_j (z) \Big[V\cdot (\mathds{1}_\lambda (P_V)f)\Big](z)\:dz \\
    &=\sum_{|\lambda_j-\lambda|\leq \epsilon} m(\lambda_j, \lambda-\epsilon) e_j (x)\langle V\cdot (\mathds{1}_\lambda (P_V)f), e_j \rangle_{L^2 (M)} \\
    &=\mathds{1}_\lambda (P)m(P, \lambda-\epsilon) (V\cdot (\mathds{1}_\lambda (P_V))f)(x).
\end{align*}
Since $\mathds{1}_\lambda (P)=\mathds{1}_\lambda (P)\circ \mathds{1}_\lambda (P)$, by \eqref{large-window-est}, \eqref{log-window-est}, \eqref{eqn:prop-(4)}, duality, and H\"older's inequality, we have, for all $2<q\leq \infty$,
\begin{align*}
    \|K_2 f\|_{L^q (M)}&=\|\mathds{1}_\lambda (P) m(P, \lambda-\epsilon) (V\cdot (\mathds{1}_\lambda (P_V)f))\|_{L^q (M)} \\
    &\lesssim \frac{\lambda^{\sigma (q)}}{(\log \lambda)^{\frac{1}{2}}} \|\mathds{1}_\lambda (P) m(P, \lambda-\epsilon) (V\cdot (\mathds{1}_\lambda (P_V)f))\|_{L^2 (M)} \\
    &\lesssim \frac{\lambda^{\sigma (q)}}{(\log \lambda)^{\frac{1}{2}}} (\lambda \epsilon)^{-1} \| \mathds{1}_\lambda (P) (V\cdot (\mathds{1}_\lambda (P_V))f)\|_{L^2 (M)} \\
    &\lesssim \frac{\lambda^{\sigma(q)}}{(\log \lambda)^{\frac{1}{2}}} (\lambda\epsilon)^{-1}\cdot \lambda^{\sigma \left(\frac{2n}{n-2}\right)}\epsilon^{\frac{1}{2}} \|V\cdot \mathds{1}_\lambda (P_V)f\|_{L^{\frac{2n}{n+2}}(M)} \\
    &\leq \frac{\lambda^{\sigma(q)}}{(\log \lambda)^{\frac{1}{2}}} (\lambda \epsilon)^{-1} \cdot \lambda^{\sigma\left(\frac{2n}{n-2} \right)} \epsilon^{\frac{1}{2}} \|V\|_{L^{\frac{n}{2}}(M)}\|\mathds{1}_\lambda (P_V) f\|_{L^{\frac{2n}{n-2}} (M)} \\
    &\lesssim \frac{\lambda^{\sigma(q)}}{(\log \lambda)^{\frac{1}{2}}} (\lambda \epsilon)^{-1} \cdot \lambda^{2\sigma\left(\frac{2n}{n-2} \right)}\epsilon \|V\|_{L^{\frac{n}{2}}(M)}\|f\|_{L^2 (M)} \\
    &=\frac{\lambda^{\sigma(q)}}{(\log \lambda)^{\frac{1}{2}}} \|V\|_{L^{\frac{n}{2}}(M)}\|f\|_{L^2 (M)},
\end{align*}
which satisfies the desired bound. This also proves the case for $V\in L^r (M)$ for $r>\frac{n}{2}$, since if $V\in L^r (M)$, then $V\in L^{\frac{n}{2}}(M)$ by the compactness of $M$.

\subsection{\texorpdfstring{Case 3: $|\lambda_j-\lambda|\leq \epsilon$, $|\tau_k-\lambda|\in (2^l, 2^{l+1}], \epsilon\leq 2^l\leq \lambda$}{Case 3}}
In this case, we have $m(\lambda_j, \tau_k)=\frac{\chi_\lambda (\lambda_j)}{\lambda_j^2-\tau_k^2}$, and
\begin{align*}
    |m(s, \tau_k)|+\epsilon |\partial_s m(s, \tau_k)|\lesssim \lambda^{-1} 2^{-l},\quad \text{for } |s-\lambda|\leq \epsilon.
\end{align*}
We write
\begin{align*}
    \sum_{|\tau_k-\lambda|\in (2^l, 2^{l+1}]} \sum_{|\lambda_j-\lambda|\leq \epsilon} \int_M m(\lambda_j, \tau_k) e_j (x) e_j (z) e_k^V (z) e_k^V (y) V(z)\:dz=K_{1, l}(x, y)+K_{2, l}(x, y),
\end{align*}
where
\begin{align*}
    & K_{1, l}(x, y)=\sum_{|\tau_k-\lambda|\in (2^l, 2^{l+1}]} \sum_{|\lambda_j-\lambda|\leq \epsilon} \int_M \int_{\lambda-\epsilon}^{\lambda+\epsilon} \partial_s m(s, \tau_k) \mathds{1}_{[\lambda-\epsilon, \lambda_j]}(s) e_j (x) e_j (z) e_k^V (z) e_k^V (y) V(z)\:dz\:ds, \\
    & K_{2, l}(x, y)=\sum_{|\tau_k-\lambda|\in (2^l, 2^{l+1}]} \sum_{|\lambda_j-\lambda|\leq \epsilon} \int_M m(\lambda-\epsilon, \tau_k) e_j (x) e_j (z) e_k^V (z) e_k^V (y) V(z)\:dz.
\end{align*}
We only estimate $K_{2, l}$, since the computation for $K_{1, l}$ is similar. For any $f\in L^2 (M)$, as in Case 1,
\begin{align*}
    \|K_2 f\|_{L^q (M)}&=\|\mathds{1}_\lambda (P) (V\cdot \mathds{1}_{\lambda, l}(P_V)m(\lambda-\epsilon, P_V)f)\|_{L^q (M)} \\
    &\lesssim \frac{\lambda^{\sigma (q)}}{(\log \lambda)^{\frac{1}{2}}} \|\mathds{1}_\lambda (P)(V\cdot \mathds{1}_{\lambda, l} (P_V)m(\lambda-\epsilon, P_V)f) \|_{L^2 (M)} \\
    &\lesssim \frac{\lambda^{\sigma(q)}}{(\log \lambda)^{\frac{1}{2}}} \cdot \lambda^{\sigma\left(\frac{2n}{n-2} \right)}\epsilon^{\frac{1}{2}} \|V\cdot \mathds{1}_{\lambda, l} (P_V) m(\lambda-\epsilon, P_V)f\|_{L^{\frac{2n}{n+2}}(M)} \\
    &\leq \frac{\lambda^{\sigma(q)}}{(\log \lambda)^{\frac{1}{2}}} \cdot \lambda^{\sigma\left(\frac{2n}{n-2} \right)}\epsilon^{\frac{1}{2}} \|V\|_{L^{\frac{n}{2}}(M)}\cdot \|\mathds{1}_{\lambda, l} (P_V)m(\lambda-\epsilon, P_V)f\|_{L^{\frac{2n}{n-2}(M)}} \\
    &\lesssim \frac{\lambda^{\sigma(q)}}{(\log \lambda)^{\frac{1}{2}}} \cdot \lambda^{2\sigma\left(\frac{2n}{n-2} \right)}\epsilon^{\frac{1}{2}}2^{\frac{l}{2}}\lambda^{-1}\cdot 2^{-l} \|V\|_{L^{\frac{n}{2}}(M)}\|f\|_{L^2 (M)} \\
    &=\frac{\lambda^{\sigma (q)}}{\log \lambda} \cdot 2^{-\frac{l}{2}} \|V\|_{L^{\frac{n}{2}(M)}} \|f\|_{L^2 (M)}.
\end{align*}
Summing over all $\epsilon\leq 2^l\leq \lambda$ gives the desired bound.

\subsection{\texorpdfstring{Case 4: $|\lambda_j-\lambda|\leq \epsilon, \tau_k>2\lambda$}{Case 4}}
In this case, we can write
\begin{align*}
    m(\lambda_j, \tau_k)=\frac{\chi_\lambda (\lambda_j)}{\lambda_j^2-\tau_k^2} &=\int_0^\infty \chi_\lambda (\lambda_j) e^{-t(\tau_k^2-\lambda_j^2)}\:dt \\
    &=:m_1 (\lambda_j, \tau_k)+m_2 (\lambda_j, \tau_k),
\end{align*}
where
\begin{align*}
    m_1 (\lambda_j, \tau_k)=\int_0^{\lambda^{-2}} \chi_\lambda (\lambda_j) e^{-t(\tau_k^2-\lambda_j^2)}\:dt,\quad m_2 (\lambda_j, \tau_k)=\frac{\chi_\lambda (\lambda_j)e^{-\lambda^{-2}(\tau_k^2-\lambda_j^2)}}{\tau_k^2-\lambda_j^2}.
\end{align*}
We first consider $m_2$. We write
\begin{align*}
    (2\lambda, \infty)=\bigcup_{l=1}^\infty I_l,\quad \text{where } I_l=(2^l \lambda, 2^{l+1}\lambda].
\end{align*}
One can check that, for $\tau_k\in I_l, |s-\lambda|\leq \epsilon$, and $N\in \mathbb{N}$,
\begin{align*}
    |m_2 (s, \tau_k)|+\epsilon|\partial_s (s, \tau_k)|\lesssim \leq C_N \lambda^{-2} 2^{-Nl}.
\end{align*}
As in Case 1, we write
\begin{align*}
    \sum_{|\lambda_j-\lambda|\leq \epsilon} \sum_{\tau_k\in I_l} \int_M m_2 (\lambda_j, \tau_k)e_j (x) e_j (z) e_k^V (z) e_k^V (y) V(z)\:dz=K_{1, l}(x, y)+K_{2, l}(x, y),
\end{align*}
where
\begin{align*}
    & K_{1, l}(x, y)=\sum_{|\lambda_j-\lambda|\leq \epsilon} \sum_{\tau_k\in I_l} \int_M \int_{\lambda-\epsilon}^{\lambda+\epsilon} \partial_s m_2 (s, \tau_k) \mathds{1}_{[\lambda-\epsilon, \lambda_j]} (s) e_j (x) e_j (z) e_k^V (z) e_k^V (y) V(z)\:dz\:ds, \\
    & K_{2, l}(x, y)=\sum_{|\lambda_j-\lambda|\leq \epsilon} \sum_{\tau_k\in I_l} \int_M m_2 (\lambda-\epsilon, \tau_k) e_j (x) e_j (z) e_k^V (z) e_k^V (y) V(z)\:dz.
\end{align*}
We only compute $K_{2, l}$, since the computation for $K_{1, l}$ is similar. As in Case 1, for any $f\in L^2 (M)$,
\begin{align*}
    \|K_{2, l}f\|_{L^q (M)}&=\|\mathds{1}_\lambda (P) (V\cdot \mathds{1}_{I_l} (P_V) m_2 (\lambda-\epsilon, P_V)f) \|_{L^q (M)} \\
    &\lesssim \frac{\lambda^{\sigma(q)}}{(\log \lambda)^{\frac{1}{2}}} \|\mathds{1}_\lambda (P)(V\cdot \mathds{1}_{I_l}(P_V)m_2 (\lambda-\epsilon, P_V)f)\|_{L^2 (M)} \\
    &\lesssim \frac{\lambda^{\sigma(q)}}{(\log \lambda)^{\frac{1}{2}}} \cdot \lambda^{\sigma\left(\frac{2n}{n-2}\right)} \epsilon^{\frac{1}{2}} \|V\cdot \mathds{1}_{I_l}(P_V)m_2 (\lambda-\epsilon, P_V)f\|_{L^{\frac{2n}{n+2}} (M)} \\
    &\lesssim \frac{\lambda^{\sigma(q)}}{(\log \lambda)^{\frac{1}{2}}} \cdot \lambda^{\sigma\left(\frac{2n}{n-2} \right)} \epsilon^{\frac{1}{2}} (\lambda 2^l)^{\sigma\left(\frac{2n}{n-2} \right)} (\lambda 2^l)^{\frac{1}{2}}\cdot \lambda^{-2}2^{-Nl} \|V\|_{L^{\frac{n}{2}}(M)} \|f\|_{L^2 (M)} \\
    &=\frac{\lambda^{\sigma(q)-1}}{\log \lambda}(2^l)^{\sigma\left(\frac{2n}{n-2} \right)-N+\frac{1}{2}} \|V\|_{L^{\frac{n}{2}}(M)} \|f\|_{L^2 (M)}.
\end{align*}
Taking $N\gg 1$ large enough, summing over all $l$ gives us a better bound.

We next consider $m_1$. We split the sum $\sum_{\tau_k>2\lambda}$ into the difference between the complete sum
\begin{align}\label{eqn:prop-(5)}
    \sum_{|\lambda_j-\lambda|\leq \epsilon} \sum_{\tau_k} \int_M m_1 (\lambda_j, \tau_k) e_j (x) e_j (z) e_k^V (z) e_k^V (y) V(z)\:dz,
\end{align}
and the partial sum
\begin{align}\label{eqn:prop-(6)}
    \sum_{|\lambda_j-\lambda|\leq \epsilon} \sum_{\tau_k\leq 2\lambda} \int_M m_1 (\lambda_j, \tau_k) e_j (x) e_j (z) e_k^V (z) e_k^V (y) V(z)\:dz.
\end{align}
We first consider the partial sum \eqref{eqn:prop-(6)}. If $\tau_k\leq 2\lambda$ and $|s-\lambda|\leq \epsilon$, then
\begin{align*}
    |m_1 (s, \tau_k)|+\epsilon|\partial_s m_1 (s, \tau_k)|\lesssim \lambda^{-2}.
\end{align*}
Then the partial sum can be written as $K_1 (x, y)+K_2 (x, y)$, where
\begin{align*}
    & K_1 (x, y)=\sum_{|\lambda_j-\lambda|\leq \epsilon} \sum_{\tau_k\leq 2\lambda} \int_M \int_{\lambda-\epsilon}^{\lambda+\epsilon} \partial_s m_1 (s, \tau_k) \mathds{1}_{[\lambda-1, \lambda_j]} (s) e_j (x) e_j (z) e_k^V (z) e_k^V (y) V(z)\:dz\:ds, \\
    & K_2 (x, y)=\sum_{|\lambda_j-\lambda|\leq \epsilon} \sum_{\tau_k\leq 2\lambda} \int_M m_1 (\lambda-\epsilon, \tau_k) e_j (x) e_j (z) e_k^V (z) e_k^V (y) V(z)\:dz\:ds.
\end{align*}
For $K_2$, as before, for any $f\in L^2 (M)$, we have
\begin{align*}
    \|K_2 f\|_{L^q(M)}&=\|\mathds{1}_\lambda (P) (V\cdot \mathds{1}_{\leq 2\lambda}(P_V) m_1 (\lambda-\epsilon, P_V) f)\|_{L^q (M)} \\
    &\lesssim \frac{\lambda^{\sigma (q)}}{(\log \lambda)^{\frac{1}{2}}} \|\mathds{1}_\lambda (P) (V\cdot \mathds{1}_{\leq 2\lambda}(P_V)m_1 (\lambda-\epsilon, P_V)f)\|_{L^2 (M)} \\
    &\lesssim \frac{\lambda^{\sigma(q)}}{(\log \lambda)^{\frac{1}{2}}}\cdot \lambda^{\sigma\left(\frac{2n}{n-2} \right)}\epsilon^{\frac{1}{2}} \|V\cdot \mathds{1}_{\leq 2\lambda} (P_V) m_1 (\lambda-\epsilon, P_V)f\|_{L^{\frac{2n}{n+2}}(M)} \\
    &\lesssim \frac{\lambda^{\sigma(q)}}{(\log \lambda)^{\frac{1}{2}}} \cdot \lambda^{\sigma\left(\frac{2n}{n-2} \right)} \epsilon^{\frac{1}{2}} \|V\|_{L^{\frac{n}{2}}(M)} \cdot \lambda^{\sigma\left(\frac{2n}{n-2} \right)} \cdot \lambda^{\frac{1}{2}} \|m_1 (\lambda-\epsilon, P_V)f\|_{L^2 (M)} \\
    &\lesssim \frac{\lambda^{\sigma(q)}}{(\log \lambda)^{\frac{1}{2}}} \cdot \lambda^{-\frac{1}{2}} \cdot \epsilon^{\frac{1}{2}}\|V\|_{L^{\frac{n}{2}}(M)} \|f\|_{L^2 (M)},
\end{align*}
which is better than what we want. The computation for $K_1$ is similar.

For the complete sum \eqref{eqn:prop-(5)}, we recall the heat kernel bounds (see, e.g., \cite[(3.6)]{HuangWangZhang2026restriction})
\begin{align}\label{eqn:prop-(7)}
    \| e^{-tH_V}\|_{L^p(M)\to L^q (M)}\lesssim t^{-\frac{n}{2}\left(\frac{1}{p}-\frac{1}{q} \right)},\quad \text{if } 0<t\leq 1 \text{ and } 1\leq p\leq q\leq \infty.
\end{align}
We also note that the operator
\begin{align*}
    \mathds{1}_\lambda (P) e^{-t\Delta_g}f(x)=\sum_{|\lambda_j-\lambda|\leq \epsilon} e^{it\lambda_j^2} e_j(x) \langle f, e_j \rangle_{L^2 (M)}
\end{align*}
is a bounded operator on $L^2 (M)$ whenever $0\leq t\leq \lambda^{-2}$. Combining this, \eqref{eqn:prop-(7)}, and the argument in Case 1, the contribution of the complete sum is majorized by a uniform constant times
\begin{align*}
    &\int_0^{\lambda^{-2}} \|\mathds{1}_\lambda (P)e^{-t\Delta_g} (V\cdot e^{-tH_V}f)\|_{L^q (M)}\:dt \\
    &\lesssim \frac{\lambda^{\sigma(q)}}{(\log \lambda)^{\frac{1}{2}}}\int_0^{\lambda^{-2}} \|\mathds{1}_\lambda (P) (V\cdot e^{-tH_V}f)\|_{L^2 (M)}\:dt \\
    &\lesssim \frac{\lambda^{\sigma(q)}}{(\log \lambda)^{\frac{1}{2}}} \cdot \lambda^{\sigma\left(\frac{2n}{n-2} \right)} \epsilon^{\frac{1}{2}} \int_0^{\lambda^{-2}} \|V\cdot e^{-tH_V}f\|_{L^{\frac{2n}{n+2}}(M)}\:dt \\
    &\leq \frac{\lambda^{\sigma(q)}}{(\log \lambda)^{\frac{1}{2}}} \cdot \lambda^{\sigma\left(\frac{2n}{n-2} \right)} \cdot \epsilon^{\frac{1}{2}} \cdot \|V\|_{L^{\frac{n}{2}}(M)} \int_0^{\lambda^{-2}} \|e^{-tH_V}f\|_{L^{\frac{2n}{n-2}(M)}}\:dt \\
    &\lesssim \frac{\lambda^{\sigma(q)}}{(\log \lambda)^{\frac{1}{2}}} \cdot\lambda^{\sigma\left(\frac{2n}{n-2} \right)}\cdot \epsilon^{\frac{1}{2}} \cdot \|V\|_{L^{\frac{n}{2}}(M)}\left(\int_0^{\lambda^{-2}} t^{-\frac{n}{2}\left(\frac{1}{2}-\frac{n-2}{2n} \right)}\:dt \right) \|f\|_{L^2 (M)} \\
    &=\frac{\lambda^{\sigma(q)-\frac{1}{2}}}{\log \lambda} \|V\|_{L^{\frac{n}{2}}(M)}\|f\|_{L^2 (M)}.
\end{align*}
This is better than what we want.

\subsection{\texorpdfstring{Case 5: $|\tau_k-\lambda|\leq \epsilon, \lambda_j>2\lambda$}{Case 5}}
In Case 5, we have
\begin{align*}
    m(\lambda_j, \tau_k)=\frac{-\chi_\lambda (\tau_k)}{\lambda_j^2-\tau_k^2}.
\end{align*}
As in Case 4, we write
\begin{align*}
    \frac{\chi_\lambda (\tau_k)}{\lambda_j^2-\tau_k^2}=m_1 (\lambda_j, \tau_k)+m_2(\lambda_j, \tau_k),
\end{align*}
where
\begin{align*}
    m_1 (\lambda_j, \tau_k)=\int_0^{\lambda^{-2}} \chi_\lambda (\tau_k) e^{-t(\lambda_j^2-\tau_k^2)}\:dt,\quad m_2 (\lambda_j, \tau_k)=\frac{\chi_\lambda (\tau_k) e^{-\lambda^{-2}(\lambda_j^2-\tau_k^2)}}{\lambda_j^2-\tau_k^2}.
\end{align*}
We first consider $m_2$ as in Case 4. If we write $(2\lambda, \infty)=\bigcup_{l=1}^\infty I_l$, where $I_l=(2^l \lambda, 2^{l+1}\lambda]$, then
\begin{align*}
    |m_2 (\lambda_j, s)|+\epsilon|\partial_s m_2 (\lambda_j, s)|\leq C_N \lambda^{-2} 2^{-Nl},\quad \text{for all } N\in \mathbb{N}.
\end{align*}
As before, we write
\begin{align*}
    \sum_{|\tau_k-\lambda|\leq \epsilon} \sum_{\lambda_j\in I_l} \int_M m_2 (\lambda_j, \tau_k) e_j (x) e_j (z) e_k^V (z) e_k^V (y) V(z)\:dz=K_{1, l}(x, y)+K_{2, l}(x, y),
\end{align*}
where
\begin{align*}
    & K_{1, l}(x, y)=\sum_{|\tau_k-\lambda|\leq \epsilon} \sum_{\lambda_j\in I_l} \int_M \int_{\lambda-\epsilon}^{\lambda+\epsilon} \partial_s m_2 (\lambda_j, s) \mathds{1}_{[\lambda-\epsilon, \tau_k]} (s) e_j (x) e_j (z) e_k^V (z) e_k^V (y) V(z)\:dz, \\
    & K_{2, l}(x, y)=\sum_{|\tau_k-\lambda|\leq \epsilon} \sum_{\lambda_j\in I_l} \int_M m_2 (\lambda_j, \lambda-\epsilon) e_j (x) e_j (z) e_k^V (z) e_k^V (y) V(z)\:dz.
\end{align*}
By the argument of Case 1 and \eqref{large-window-est},
\begin{align*}
    \|K_{2, l}f\|_{L^q (M)}&=\|\mathds{1}_{I_l}(P) m_2 (P, \lambda-\epsilon) (V\cdot \mathds{1}_\lambda (P_V)f)\|_{L^q (M)} \\
    &\lesssim (\lambda 2^l)^{\sigma(q)} (\lambda 2^l)^{\frac{1}{2}} \| \mathds{1}_{I_l} (P) m_2 (P, \lambda-\epsilon) (V\cdot \mathds{1}_\lambda (P_V)f) \|_{L^2 (M)} \\
    &\lesssim (\lambda 2^l)^{\sigma(q)} (\lambda 2^l)^{\frac{1}{2}} \cdot \lambda^{-2}\cdot 2^{-Nl} \|\mathds{1}_{I_l}(P) (V\cdot \mathds{1}_\lambda (P_V)f)\|_{L^2 (M)} \\
    &\lesssim (\lambda 2^l)^{\sigma(q)} (\lambda 2^l)^{\frac{1}{2}} \cdot \lambda^{-2}\cdot 2^{-Nl} (\lambda 2^l)^{\sigma\left(\frac{2n}{n-2} \right)+\frac{1}{2}} \|V\cdot \mathds{1}_\lambda (P_V)f\|_{L^{\frac{2n}{n+2}} (M)} \\
    &\lesssim (\lambda 2^l)^{\sigma(q)} (\lambda 2^l)^{\frac{1}{2}} \cdot \lambda^{-2}\cdot 2^{-Nl} (\lambda 2^l)^{\sigma\left(\frac{2n}{n-2} \right)+\frac{1}{2}}\|V\|_{L^{\frac{n}{2}}(M)}\cdot \lambda^{\sigma\left(\frac{2n}{n-2} \right)}\epsilon^{\frac{1}{2}} \|f\|_{L^2 (M)} \\
    &=\frac{\lambda^{\sigma(q)}}{(\log \lambda)^{\frac{1}{2}}} \cdot (2^l)^{\sigma(q)+\frac{3}{2}-N}\|V\|_{L^{\frac{n}{2}}(M)}\|f\|_{L^2 (M)}.
\end{align*}
Summing over all $l$ with $N\gg 1$ sufficiently large gives the desired bound. The computation for $K_{1, l}$ is similar.

To handle $m_1$, as in Case 4, we divide the sum $\sum_{\lambda_j>2\lambda}$ into the difference between the complete sum
\begin{align}\label{eqn:prop-(8)}
    \sum_{|\tau_k-\lambda|\leq \epsilon} \sum_{\lambda_j} \int_M m_1 (\lambda_j, \tau_k) e_j (x) e_j (z) e_k^V (z) e_k^V (y) V(z)\:dz
\end{align}
and the partial sum
\begin{align}\label{eqn:prop-(9)}
    \sum_{|\tau_k-\lambda|\leq \epsilon} \sum_{\lambda_j\leq 2\lambda} \int_M m_1 (\lambda_j, \tau_k) e_j (x) e_j (z) e_k^V (z) e_k^V (y) V(z)\:dz.
\end{align}
We now consider the partial sum \eqref{eqn:prop-(9)} as before. If $\lambda_j\leq 2\lambda$ and $|s-\lambda|\leq \epsilon$, we have
\begin{align*}
    |m_1 (\lambda_j, s)|+\epsilon|\partial_s m_1 (\lambda_j, s)|\lesssim \lambda^{-2}.
\end{align*}
We write \eqref{eqn:prop-(9)} as $K_1 (x, y)+K_2 (x, y)$, where
\begin{align*}
    & K_1 (x, y)=\sum_{|\tau_k-\lambda|\leq \epsilon} \sum_{\lambda_j\leq 2\lambda} \int_M \int_{\lambda-\epsilon}^{\lambda+\epsilon} \partial_s m_1 (\lambda_j, s) \mathds{1}_{[\lambda-\epsilon, \tau_k]}(s) e_j (x) e_j (z) e_k^V (z) e_k^V (y) V(z)\:dz\:ds, \\
    & K_2 (x, y)=\sum_{|\tau_k-\lambda|\leq \epsilon} \sum_{\lambda_j\leq 2\lambda} \int_M m_1 (\lambda_j, \lambda-\epsilon) e_j (x) e_j (z) e_k^V (z) e_k^V (y) V(z)\:dz.
\end{align*}
By the argument in Case 1 and \eqref{large-window-est}, for any $f\in L^2 (M)$,
\begin{align*}
    \|K_2 f\|_{L^q (M)}&=\|\mathds{1}_{\leq 2\lambda} (P) m_1 (P, \lambda-\epsilon) (V\cdot \mathds{1}_\lambda (P_V)f)\|_{L^q (M)} \\
    &\lesssim \lambda^{\sigma(q)} \cdot \lambda^{\frac{1}{2}} \|\mathds{1}_{\leq 2\lambda} (P) m_1 (P, \lambda-\epsilon) (V\cdot \mathds{1}_\lambda (P_V)f)\|_{L^2 (M)} \\
    &\lesssim \lambda^{\sigma(q)+\frac{1}{2}}\cdot \lambda^{-2} \|\mathds{1}_{\leq 2\lambda} (P) (V\cdot \mathds{1}_\lambda (P_V)f)\|_{L^2 (M)} \\
    &\lesssim \lambda^{\sigma(q)-\frac{3}{2}}\cdot \lambda^{\sigma\left(\frac{2n}{n-2} \right)} \cdot \lambda^{\frac{1}{2}} \|V\cdot \mathds{1}_\lambda (P_V)f\|_{L^{\frac{2n}{n+2}} (M)} \\
    &\lesssim \lambda^{\sigma(q)-\frac{1}{2}} \|V\|_{L^{\frac{n}{2}}(M)} \|\mathds{1}_\lambda (P_V)f\|_{L^{\frac{2n}{n-2}}(M)} \\
    &\lesssim \lambda^{\sigma(q)-\frac{1}{2}}\|V\|_{L^{\frac{n}{2}}(M)} \cdot \lambda^{\sigma\left(\frac{2n}{n-2} \right)}\epsilon^{\frac{1}{2}}\|f\|_{L^2 (M)} \\
    &=\frac{\lambda^{\sigma(q)}}{(\log \lambda)^{\frac{1}{2}}} \|V\|_{L^{\frac{n}{2}}(M)}\|f\|_{L^2 (M)},
\end{align*}
as desired. The computation for $K_1$ is similar.

For the complete sum \eqref{eqn:prop-(8)}, we use the heat kernel Gaussian bounds to calculate the kernel of $m_1 (P, s)$ with $|s-\lambda|\leq \epsilon$ (see, e.g., \cite[(2.1)]{HuangWangZhang2026restriction} and \cite[computations in p.15]{HuangWangZhang2026restriction})
\begin{align*}
    \begin{split}
        \left|\sum_{\lambda_j} m_1 (\lambda_j, s) e_j (x) e_j (y) \right|&\lesssim \int_0^{\lambda^{-2}} \left|\sum_{\lambda_j} e^{-t\lambda_j^2} e_j (x) e_j (y) \right|\:dt \\
        &\lesssim \int_0^{\lambda^{-2}} t^{-\frac{n}{2}} e^{-c d_g (x, y)^2/t}\:dt \\
        &\lesssim \begin{cases}
            \log (2+(\lambda d_g (x, y))^{-1}) (1+\lambda d_g (x, y))^{-N}, & n=2, \\
            d_g (x, y)^{2-n} (1+\lambda d_g (x, y))^{-N}, & n\geq 3,
        \end{cases} \\
        &\lesssim W_n (d_g (x, y)) (1+\lambda d_g (x, y))^{-N},\quad \text{for all } N\in \mathbb{N},
    \end{split}
\end{align*}
where $W_n$ is as in \eqref{setup-Wn}, and so, for $|s-\lambda|\leq \epsilon$,
\begin{align}\label{eqn:prop-(10)}
    |m_1 (P, s) (x, y)|\lesssim W_n (d_g (x, y))(1+\lambda d_g (x, y))^{-N},\quad N\in \mathbb{N}.
\end{align}
Similarly, one can compute, for $|s-\lambda|\leq \epsilon$,
\begin{align*}
    |\partial_s m_1 (P, s)(x, y)|\lesssim W_n (d_g (x, y)) (1+\lambda d_g (x, y))^{-N},\quad \text{for all } N\in \mathbb{N}.
\end{align*}
We write the complete sum \eqref{eqn:prop-(8)} as $K_1 (x, y)+K_2 (x, y)$, where
\begin{align*}
    & K_1 (x, y)=\sum_{|\tau_k-\lambda|\leq \epsilon} \sum_{\lambda_j} \int_M \int_{\lambda-\epsilon}^{\lambda+\epsilon} \partial_s m_1 (\lambda_j, \tau_k) \mathds{1}_{[\lambda-\epsilon, \tau_k]}(s) e_j (x) e_j (z) e_k^V (z) e_k^V (y) V(z)\:dz\:ds, \\
    & K_2 (x, y)=\sum_{|\tau_k-\lambda|\leq \epsilon} \sum_{\lambda_j} \int_M m_1 (\lambda_j, \tau_k) e_j (x) e_j (z) e_k^V (z) e_k^V (y) V(z)\:dz.
\end{align*}
Again, we focus on $K_2$, since the computation for $K_1$ is similar. As in \cite{HuangWangZhang2026restriction}, we dyadically decompose the kernel $m_1 (P, \lambda-\epsilon)(x, y)$ with respect to $d_g (x, y)$ and apply Young's inequality to each piece
\begin{align*}
    m_{1, l}(P, \lambda-\epsilon):=m_1 (P, \lambda-\epsilon) (x, y)\mathds{1}_{d_g (x, y)\approx 2^{-l}}(x, y)
\end{align*}
so that we can obtain, for all $N\in \mathbb{N}$,
\begin{align*}
    \|m_{1, l}(P, \lambda-\epsilon)\|_{L^{p_1}(M)\to L^q (M)}\lesssim 2^{\left(\frac{n}{p_1}-\frac{n}{q}-2\right)l} (1+\lambda 2^{-l})^{-N},\quad \text{where } 1\leq p_1\leq q\leq \infty.
\end{align*}
Summing over the dyadic decomposition, we obtain
\begin{align}\label{ineq:prop-(13)}
    \|m_1 (P, \lambda-\epsilon)\|_{L^{p_1}(M)\to L^q (M)}\lesssim \lambda^{\frac{n}{p_1}-\frac{n}{q}-2}, \quad\text{where } \frac{n}{p_1}-\frac{n}{q}-2<0 \text{ and } 1\leq p_1\leq q\leq \infty.
\end{align}

\subsubsection{\texorpdfstring{Case 5-1: $n=2$ and $V\in \mathcal{K}(M)$}{Case 5-1}}
We recall that Case 5-1 of Proposition \ref{prop:hv-Kato-and-Ln2} for $6<q\leq \infty$ is already proved in \cite[Theorem 5.1]{BlairHuangSireSogge2022UniformSobolev}, so here, we can focus on $2<q\leq 6$. Taking $(n, p_1)=(2,1)$ in \eqref{ineq:prop-(13)}, we have
\begin{align*}
    \|K_2 f\|_{L^q (M)}&=\|m_1 (P, \lambda-\epsilon)(V\cdot \mathds{1}_\lambda (P_V)f)\|_{L^q (M)} \\
    &\lesssim \lambda^{2-\frac{2}{q}-2}\|V\cdot \mathds{1}_\lambda (P_V)f)\|_{L^1 (M)}\\
    &\lesssim \lambda^{-\frac{2}{q}}\|V\|_{L^1(M)} \|\mathds{1}_\lambda (P_V)f\|_{L^{\frac{2n}{n-2}}(M)} \\
    &\lesssim \frac{\lambda^{\frac{1}{2}-\frac{2}{q}}}{(\log \lambda)^{\frac{1}{2}}} \|V\|_{L^1 (M)} \|f\|_{L^2 (M)}\leq \frac{\lambda^{\sigma(q)}}{(\log \lambda)^{\frac{1}{2}}}\|V\|_{L^1 (M)}\|f\|_{L^2 (M)},
\end{align*}
as desired.

\subsubsection{\texorpdfstring{Case 5-2: $n\in \{3, 4\}$, $V\in \mathcal{K}(M)\cap L^{\frac{n}{2}}(M)$, and $2<q<\infty$}{Case 5-2}}
The computation for Case 5-2 is similar to that of Case 5-1, except $q=\infty$. The $q=\infty$ for $n=2$ was treated well in \cite[Theorem 5.1]{BlairHuangSireSogge2022UniformSobolev} by using the perturbation argument (to be specific, see \cite[(5.10)]{BlairHuangSireSogge2022UniformSobolev}), but the computation may not be well translated into the case where $n\geq 3$, so we focus on the exponents $2<q<\infty$ for $n\geq 3$ in Case 5-2.

For $n\in \{3, 4\}$, by \cite[Theorem 1.3]{BlairHuangSireSogge2022UniformSobolev}, we know that Proposition \ref{prop:hv-Kato-and-Ln2} holds for $\frac{2(n+1)}{n-1}<q<\infty$, so we can focus on $2<q\leq \frac{2n}{n-2}$. If we take $p_1=\frac{2n}{n+2}$ and $2<q<\frac{2n}{n-2}$, then by \eqref{ineq:prop-(13)},
\begin{align}\label{case5-2-m1-comp}
    \begin{split}
        \|m_1 (P, \lambda-\epsilon) (V\cdot \mathds{1}_\lambda (P_V)f)\|_{L^q (M)}&\lesssim \lambda^{\frac{n+2}{2}-\frac{n}{q}-2} \|V\cdot \mathds{1}_\lambda (P_V)f\|_{L^{\frac{2n}{n+2}}(M)} \\
        &\leq \lambda^{\frac{n-2}{2}-\frac{n}{q}} \|V\|_{L^{\frac{n}{2}}(M)} \|\mathds{1}_\lambda (P_V)f\|_{L^{\frac{2n}{n-2}}(M)} \\
        &\lesssim \frac{\lambda^{\frac{n-1}{2}-\frac{n}{q}}}{(\log \lambda)^{\frac{1}{2}}} \|V\|_{L^{\frac{n}{2}}(M)} \|f\|_{L^2 (M)} \\
        &\leq \frac{\lambda^{\sigma(q)}}{(\log \lambda)^{\frac{1}{2}}}\|V\|_{L^{\frac{n}{2}}(M)} \|f\|_{L^2 (M)},
    \end{split}
\end{align}
as desired.

\subsubsection{\texorpdfstring{Case 5-3: $n\geq 5$, $V\in \mathcal{K}(M) \cap L^{\frac{n}{2}}(M)$, and $2<q\leq \frac{2n}{n-4}$}{Case 5-3}}\label{sec-case-5-3}
Again, for $n\geq 5$, by \cite[Theorem 1.3]{BlairHuangSireSogge2022UniformSobolev}, we know that the estimate in Proposition \ref{prop:hv-Kato-and-Ln2} holds for $\frac{2(n+1)}{n-1}<q\leq \frac{2n}{n-4}$, and so, we can focus on $2<q<\frac{2n}{n-2}$. The computation is the same as in \eqref{case5-2-m1-comp}.

\subsubsection{\texorpdfstring{Case 5-4: $n\geq 3$, $V\in L^r(M)$ with $r>\frac{n}{2}$, and $2<q\leq \infty$}{Case 5-4}}
For $n\geq 3$, by Case 5-2 and Case 5-3, we know
\begin{align}\label{ineq:prop-(14)}
    \|m_1 (P, \lambda-\epsilon)(V\cdot \mathds{1}_\lambda (P_V)f)\|_{L^q (M)}\lesssim \frac{\lambda^{\sigma(q)}}{(\log\lambda)^{\frac{1}{2}}} \|f\|_{L^2 (M)},\quad 2<q<\frac{2n}{n-2}.
\end{align}
Setting $q=\frac{2(n+1)}{n-1}$ in \eqref{ineq:prop-(14)}, we have
\begin{align}\label{ineq:prop-(15)}
    \|K_2 f\|_{L^{\frac{2(n+1)}{n-1}}(M)}=\|m_1 (P, \lambda-\epsilon)(V\cdot \mathds{1}_\lambda (P_V)f)\|_{L^{\frac{2(n+1)}{n-1}}(M)}\lesssim \frac{ \lambda^{\frac{n-1}{2(n+1)}}}{(\log \lambda)^{\frac{1}{2}}}\|f\|_{L^2 (M)}.
\end{align}
On the other hand, if we take $r=p_1>\frac{n}{2}$ and $q=\infty$, then by \eqref{ineq:prop-(13)} and \eqref{large-window-est},
\begin{align}\label{ineq:m1-L-infty-est}
    \begin{split}
        \|K_2f\|_{L^\infty (M)}=\|m_1(P, \lambda-\epsilon)(V\cdot \mathds{1}_\lambda (P_V)f)\|_{L^\infty (M)}&\lesssim \lambda^{\frac{n}{p_1}-2}\|V\cdot \mathds{1}_\lambda (P_V)f\|_{L^{p_1}(M)} \\
    &\leq \lambda^{\frac{n}{p_1}-2} \|V\|_{L^{p_1}(M)} \|\mathds{1}_\lambda (P_V)f\|_{L^\infty (M)} \\
    &\leq \lambda^{\frac{n}{p_1}-2} \|V\|_{L^{p_1}(M)} \|\mathds{1}_{[\lambda-1, \lambda+1]} (P_V)f\|_{L^\infty (M)} \\
    &\lesssim \lambda^{\frac{n-1}{2}+\frac{n}{p_1}-2} \|V\|_{L^{p_1}(M)}\|f\|_{L^2(M)} \\
    &\leq \frac{\lambda^{\frac{n-1}{2}}}{(\log \lambda)^{\frac{1}{2}}} \|V\|_{L^{p_1}(M)}\|f\|_{L^2(M)} \\
    &=\frac{\lambda^{\sigma(\infty)}}{(\log \lambda)^{\frac{1}{2}}} \|V\|_{L^{p_1}(M)}\|f\|_{L^2(M)}.
    \end{split}
\end{align}
Interpolating \eqref{ineq:prop-(15)} and \eqref{ineq:m1-L-infty-est}, we have
\begin{align*}
    \|m_1 (P, \lambda-\epsilon)(V\cdot \mathds{1}_\lambda (P_V)f)\|_{L^q (M)}\lesssim \frac{\lambda^{\sigma(q)}}{(\log \lambda)^{\frac{1}{2}}} \|f\|_{L^2 (M)},\quad \frac{2(n+1)}{n-1}\leq q\leq \infty.
\end{align*}
Combining this and \eqref{ineq:prop-(14)}, we have the desired bound for all $2<q\leq \infty$. This completes Case 5. We are left to consider Case 2.

\subsection{\texorpdfstring{Case 2: $|\tau_k-\lambda|\leq \epsilon, |\lambda_j-\lambda|\in (2^l, 2^{l+1}], \epsilon\leq 2^l\leq \lambda$}{Case 2}}
We choose a cutoff function $\psi\in C_0^\infty (\mathbb{R})$ such that
\begin{align}\label{setup-psi}
    \psi(t)=1 \text{ if } |t|\leq 2,\quad \text{and} \quad \psi(t)=0 \text{ if } |t|\geq 3.
\end{align}
In Case 2, we write
\begin{align*}
    m(\lambda_j, \tau_k)=-\frac{\chi_\lambda (\tau_k)}{\lambda_j^2-\tau_k^2} \psi(\lambda_j/\lambda).
\end{align*}

\subsubsection{\texorpdfstring{Case 2-1: $n\geq 3$, $V\in L^r (M)$ with $r>\frac{n}{2}$, and $2<q\leq \infty$}{Case 2-1}}
This case can be handled as in Case 1. Indeed, we write
\begin{align*}
    \sum_{|\lambda_j-\lambda|\in (2^l, 2^{l+1}]} \sum_{|\tau)k-\lambda|\lesssim\epsilon} \int_M m(\lambda_j, \tau_k) e_j (x) e_j (z) e_k^V (z) e_k^V (y)V(z)\:dz=K_{1, l}(x, y)+K_{2, l} (x, y),
\end{align*}
where
\begin{align*}
    & K_{1, l}(x, y)=\sum_{|\lambda_j-\lambda|\in (2^l, 2^{l+1}]} \sum_{|\tau)k-\lambda|\lesssim\epsilon} \int_M \int_{\lambda-\epsilon}^{\lambda+\epsilon} \partial_s m(\lambda_j, \tau_k) \mathds{1}_{[\lambda-\epsilon, \tau_k]} (s) e_j (x) e_j (z) e_k^V (z) e_k^V (y)V(z)\:ds\:dz, \\
    & K_{2, l}(x, y)=\sum_{|\lambda_j-\lambda|\in (2^l, 2^{l+1}]} \sum_{|\tau)k-\lambda|\lesssim\epsilon} \int_M m(\lambda_j, \lambda-\epsilon) e_j (x) e_j (z) e_k^V (z) e_k^V (y)V(z)\:dz.
\end{align*}
For $|\lambda_j-\lambda|\in (2^l, 2^{l+1}]$, we have
\begin{align*}
    |m(\lambda_j, s)|+\epsilon|\partial_s m(\lambda_j, s)|\lesssim \lambda^{-1} 2^{-l}.
\end{align*}
As before, we focus on $K_{2, l}$. For any $f\in L^2(M)$, by \eqref{large-window-est} and \eqref{log-window-est}, we have
\begin{align*}
    \|K_{2, l}f\|_{L^q (M)}&=\|\mathds{1}_{\lambda, l}(P)m(P, \lambda-\epsilon)(V\cdot \mathds{1}_\lambda (P_V)f)\|_{L^q (M)} \\
    &\lesssim \lambda^{\sigma(q)} 2^{\frac{l}{2}} \|\mathds{1}_{\lambda, l}(P) m (P, \lambda-\epsilon) (V\cdot \mathds{1}_\lambda (P_V)f)\|_{L^2 (M)} \\
    &\lesssim \lambda^{\sigma(q)} 2^{\frac{l}{2}} (\lambda^{-1}2^{-l})\|\mathds{1}_{\lambda, l}(P) (V\cdot \mathds{1}_\lambda (P_V)f)\|_{L^2 (M)} \\
    &\lesssim \lambda^{\sigma(q)} 2^{\frac{l}{2}} (\lambda^{-1}2^{-l}) \lambda^{\sigma\left(\frac{2n}{n-2}\right)}2^{l/2}\|V\cdot \mathds{1}_\lambda (P_V)f\|_{L^{\frac{2n}{n+2}} (M)} \\
    &=\lambda^{\sigma(q)-\frac{1}{2}}\|V\cdot \mathds{1}_\lambda (P_V)f\|_{L^{\frac{2n}{n+2}} (M)} \\
    &\leq \lambda^{\sigma(q)-\frac{1}{2}+\sigma\left(\frac{2nr}{(n+2)r-2n}\right)} \|V\|_{L^r (M)} \|f\|_{L^2 (M)},\quad r>\frac{n}{2}.
\end{align*}
Let $q^*=\frac{2nr}{(n+2)r-2n}$. One can simply check that $q^*<\frac{2n}{n-2}$ if and only if $r>\frac{n}{2}$, and so, we have $-\frac{1}{2}+\sigma\left(\frac{2nr}{(n+2)r-2n}\right)<-\epsilon'$ for any small $\epsilon'>0$. We thus have that
\begin{align*}
    \|K_{2, l}f\|_{L^q (M)}&\lesssim \lambda^{\sigma(q)-\epsilon'} \|V\|_{L^r (M)} \|f\|_{L^2 (M)},\quad r>\frac{n}{2},\quad \epsilon'>0.
\end{align*}
Thanks to $\lambda^{-\epsilon'}$, even if we sum over $\epsilon\leq 2^l\leq \lambda$, we have a better bound when $r>\frac{n}{2}$. This proves Case 2-1.

We note that this argument cannot be appliable when $r=\frac{n}{2}$. Indeed, if $r=\frac{n}{2}$, then $\epsilon'=0$, and this is not helpful to remove the log-loss for $V\in \mathcal{K}(M)\cap L^{\frac{n}{2}}(M)$, so we need a different argument when $V\in \mathcal{K}(M)\cap L^{\frac{n}{2}} (M)$. To this end, we combine the arguments in \cite{HuangWangZhang2026restriction} and uniform Sobolev-type arguments used in the existing literature.

\subsubsection{\texorpdfstring{Case 2-2: Preliminary reduction to $n\geq 2$, and $V\in \mathcal{K}(M) \cap L^{\frac{n}{2}} (M)$}{Case 2-2}}
We first recall that the cases
\begin{itemize}
    \item $n=2$, $V\in \mathcal{K}(M)$, and $6\leq q\leq \infty$,
    \item $n\in \{3, 4\}$, $V\in L^{\frac{n}{2}}(M)$, and $\frac{2(n+1)}{n-1}\leq q<\infty$,
    \item $n\geq 5$, $V\in L^{\frac{n}{2}}(M)$, and $\frac{2(n+1)}{n-1}\leq q\leq \frac{2n}{n-4}$
\end{itemize}
have already been considered in \cite{BlairHuangSireSogge2022UniformSobolev} (and for critical exponent $q=\frac{2(n+1)}{n-1}$, we combine the arguments of \cite{BlairHuangSireSogge2022UniformSobolev} and \cite{HuangSogge2025Curvature, BlairHuangSogge2022Improved, BlairSogge2019logarithmic}), and this is a reason why we can focus on the following cases.
\begin{itemize}
    \item Case 2-2-1: $n=2$, $V\in \mathcal{K}(M)$, and $2<q<6$, which will be proved in \S \ref{subsec-Case-2-2-1}, and
    \item Case 2-2-2: $n\geq 3$, $V\in \mathcal{K}(M)\cap L^{\frac{n}{2}}(M)$, and $\frac{2n}{n-1}<q<\frac{2n}{n-3}$, which will be proved in \S \ref{sec-case-2-2-2}.
\end{itemize}
In this subsection, we collect ingredients to handle Case 2-2-1 and Case 2-2-2. We write
\begin{align*}
    m_1 (\lambda_j, \tau_k)=\frac{-\chi_\lambda (\tau_k)}{\lambda_j^2-\tau_k^2+i\lambda} \psi(\lambda_j/\lambda),\quad m_2 (\lambda_j, \tau_k)=m(\lambda_j, \tau_k)-m_1(\lambda_j, \tau_k).
\end{align*}
We first consider $m_2$. If $|s-\lambda|\leq \epsilon$ and $|\lambda_j-\lambda|\in (2^l, 2^{l+1}]$, we have
\begin{align*}
    |m_2 (\lambda_j, s)|+\epsilon|\partial_s m_2 (\lambda_j, s)|\lesssim \lambda^{-1} 2^{-2l}.
\end{align*}
As before, we can write
\begin{align*}
    \sum_{|\lambda_j-\lambda|\in (2^l, 2^{l+1}]} \sum_{|\tau_k-\lambda|\leq \epsilon} \int_M m_2 (\lambda_j, \tau_k) e_j (x) e_j (z) e_k^V (z) e_k^V(y) V(z)\:dz=K_{1, l}(x, y)+K_{2, l}(x, y),
\end{align*}
where
\begin{align*}
    & K_{1, l}(x, y)=\sum_{|\lambda_j-\lambda|\in (2^l, 2^{l+1}]} \sum_{|\tau_k-\lambda|\leq \epsilon} \int_M \int_{\lambda-\epsilon}^{\lambda+\epsilon} \partial_s m_2 (\lambda_j, s) \mathds{1}_{[\lambda-\epsilon, \tau_k]}(s) e_j (x) e_j (z) e_k^V (z) e_k^V(y) V(z)\:dz\:ds, \\
    & K_{2, l}(x, y)=\sum_{|\lambda_j-\lambda|\in (2^l, 2^{l+1}]} \sum_{|\tau_k-\lambda|\leq \epsilon} \int_M m_2 (\lambda_j, \lambda-\epsilon) e_j (x) e_j (z) e_k^V (z) e_k^V(y) V(z)\:dz.
\end{align*}
We only estimate $K_{2, l}$, since the computation for $K_{1, l}$ is similar. For any $f\in L^2 (M)$, by the argument in Case 1,
\begin{align*}
    \|K_{2, l}f\|_{L^q (M)}&=\|\mathds{1}_{\lambda, l} (P) m_2(P, \lambda-\epsilon) (V\cdot \mathds{1}_\lambda (P_V)f)\|_{L^q (M)} \\
    &\lesssim \lambda^{\sigma(q)}(2^l)^{\frac{1}{2}} \|\mathds{1}_{\lambda, l}(P) m_2 (P, \lambda-\epsilon) (V\cdot \mathds{1}_\lambda (P_V)f)\|_{L^2 (M)} \\
    &\lesssim \lambda^{\sigma(q)} (2^l)^{\frac{1}{2}}\cdot \lambda^{-1}2^{-2l} \|\mathds{1}_{\lambda, l}(P) (V\cdot \mathds{1}_\lambda (P_V)f)\|_{L^2 (M)} \\
    &\lesssim \lambda^{\sigma(q)}\cdot \lambda^{-1}2^{-\frac{3l}{2}}\cdot \lambda^{\sigma\left(\frac{2n}{n-2} \right)}2^{\frac{l}{2}} \|V\cdot \mathds{1}_\lambda (P_V)f\|_{L^{\frac{2n}{n+2}}(M)} \\
    &\leq \lambda^{\sigma(q)-\frac{1}{2}}\cdot 2^{-l} \|V\|_{L^{\frac{n}{2}}(M)} \|\mathds{1}_\lambda (P_V)f\|_{L^{\frac{2n}{n-2}}(M)} \\
    &\lesssim \frac{\lambda^{\sigma(q)}}{(\log \lambda)^{\frac{1}{2}}} \cdot 2^{-l}\|V\|_{L^{\frac{n}{2}}(M)} \|f\|_{L^2 (M)}.
\end{align*}
Summing over all $l$ gives the desired bound.

We next consider $m_1$. We consider all $\lambda_j\leq 3\lambda$ here, which follows from the condition of Case 2 (and from the support property of $\psi$). We write
\begin{align*}
    \sum_{\lambda_j\leq 3\lambda} \sum_{|\tau_k-\lambda|\leq \epsilon} \int_M m_1 (\lambda_j, \tau_k) e_j (x) e_j(z) e_k^V (z) e_k^V (y) V(z)\:dz=K_1 (x, y)+K_2 (x, y),
\end{align*}
where
\begin{align*}
    & K_1 (x, y)=\sum_{\lambda_j\leq 3\lambda} \sum_{|\tau_k-\lambda|\leq \epsilon} \int_M \int_{\lambda-\epsilon}^{\lambda+\epsilon} \partial_s m_1 (\lambda_j, s) \mathds{1}_{[\lambda-\epsilon, \tau_k]}(s) e_j (x) e_j(z) e_k^V (z) e_k^V (y) V(z)\:dz\:ds, \\
    & K_2 (x, y)=\sum_{\lambda_j\leq 3\lambda} \sum_{|\tau_k-\lambda|\leq \epsilon} \int_M m_1 (\lambda_j, \lambda-\epsilon) e_j (x) e_j(z) e_k^V (z) e_k^V (y) V(z)\:dz.
\end{align*}
As before, we focus on estimating $K_2$. Let $\psi_1\in C_0^\infty (\mathbb{R})$ be such that
\begin{align}\label{setup-psi-1}
    \psi_1 (t)=1 \text{ for } |t|\leq 3,\quad \text{and} \quad \psi_1 (t)=0 \text{ for } |t|\geq 4.
\end{align}
Suppose we have
\begin{align}\label{ineq:prop-(16)}
    \|(-\Delta_g-(\lambda+i\epsilon)^2)^{-1} \psi_1(P/\lambda)\|_{L^{\frac{2n}{n+2}}(M)\to L^q (M)}\lesssim \lambda^{\sigma(q)-\frac{1}{2}}(\log \lambda)^{\epsilon'},\quad \text{for any } \epsilon'>0.
\end{align}
One may check that the symbol $\xi\mapsto\frac{\xi^2-(\lambda+i\epsilon)^2}{\xi^2-s^2+i\lambda}\psi(\xi/\lambda)$ belongs to the symbol class $S^0$, and so, by \cite[Corollary 4.3.2]{Sogge2017FourierBook}, one would have $\|(-\Delta_g-(\lambda+i\epsilon)^2)m_1(P, \lambda-\epsilon)\|_{L^q (M)\to L^q (M)}\leq C_q$. We also note that by \eqref{setup-psi} and \eqref{setup-psi-1}, we have $\psi(t)(1-\psi_1(t))=0$ for any $t\in \mathbb{R}$, and so, the contribution of the operator $\psi(P/\lambda)(1-\psi_1 (P/\lambda))$ may be negligible. By construction, $m_1 (P, \lambda-\epsilon)$ contains $\psi(P/\lambda)$ in it, and so, this implies that the contribution of $(1-\psi_1(P/\lambda))m_1 (P, \lambda-\epsilon)$ may also be negligible. With this in mind, if \eqref{ineq:prop-(16)} is true for $n\geq 3$, then
\begin{align}\label{ineqs:result-(16)}
    \begin{split}
        \|m_1 (P, \lambda-\epsilon)(V\cdot \mathds{1}_\lambda (P_V)f)\|_{L^q (M)}&\lesssim \|m_1 (P, \lambda-\epsilon) \psi_1 (P/\lambda)(V\cdot \mathds{1}_\lambda (P_V)f)\|_{L^q (M)} \\
        &\lesssim \|(-\Delta_g-(\lambda+i\epsilon)^2)^{-1}\psi_1 (P/\lambda)(V\cdot \mathds{1}_\lambda (P_V)f)\|_{L^q (M)}\\
        &\lesssim \lambda^{\sigma(q)-\frac{1}{2}}(\log \lambda)^{\epsilon'} \|V\cdot \mathds{1}_\lambda (P_V)f\|_{L^{\frac{2n}{n+2}}(M)} \\
        &\leq \lambda^{\sigma(q)-\frac{1}{2}}(\log \lambda)^{\epsilon'} \|V\|_{L^{\frac{n}{2}}(M)}\|\mathds{1}_\lambda (P_V)f\|_{L^{\frac{2n}{n-2}}(M)} \\
        &\lesssim \lambda^{\sigma(q)-\frac{1}{2}}(\log \lambda)^{\epsilon'}\epsilon^{\frac{1}{2}} \|V\|_{L^{\frac{n}{2}}}\lambda^{\sigma\left(\frac{2n}{n-2} \right)}\|f\|_{L^2 (M)} \\
        &=\frac{\lambda^{\sigma(q)}}{(\log \lambda)^{\frac{1}{2}-\epsilon'}} \|V\|_{L^{\frac{n}{2}}}\|f\|_{L^2 (M)},
    \end{split}
\end{align}
as desired. Some cases can be treated by \eqref{ineq:prop-(16)}, but in other cases, we need different estimates, which will be explained below for $n=2$ and for some ``small'' exponents, say, $\frac{2n}{n-1}<q<\frac{2n}{n-3}$ in Case 2-2-2, \S \ref{sec-case-2-2-2}

We first find and show cases where \eqref{ineq:prop-(16)} is helpful. As in \cite{BourgainShaoSoggeYao2015Resolvent} and \cite{BlairHuangSireSogge2022UniformSobolev}, we can write
\begin{align*}
    (-\Delta_g-(\lambda+i\epsilon)^2)^{-1}=\frac{i}{\lambda+i\epsilon} \int_0^\infty e^{i\lambda t} e^{-\epsilon t} \cos (tP)\:dt.
\end{align*}
Let $\rho\in C_0^\infty (\mathbb{R})$ be such that
\begin{align*}
    \mathds{1}_{\left[-\frac{\epsilon_0}{2}, \frac{\epsilon_0}{2}\right]}\leq \rho\leq \mathds{1}_{[-\epsilon_0, \epsilon_0]},\quad \epsilon_0=\min\left\{1, \frac{1}{2}\mathrm{Inj}(M)\right\},
\end{align*}
where $\mathrm{Inj}(M)$ is the injectivity radius of $M$. We choose $\beta\in C_0^\infty (\mathbb{R})$ so that we can write
\begin{align*}
    |\beta(t)|\leq 1,\quad \mathrm{supp}(\beta)\subset [1/2, 2],\quad \sum_{j\in \mathbb{Z}} \beta(2^{-j}t)=1 \text{ for } t>0.
\end{align*}
We let
\begin{align*}
    \beta_0 (t)=1-\sum_{j=1}^\infty \beta(2^{-j}t) \text{ for } t>0,\quad \mathrm{supp}(\beta_0)\subset [-4, 4].
\end{align*}
We set
\begin{align*}
    & T_0 (\tau)=\frac{i}{\lambda+i\epsilon}\int_0^\infty \beta_0 (\lambda t) \rho(\epsilon t) e^{i\lambda t} e^{-\epsilon t} \cos (t\tau)\:dt, \\
    & T_j (\tau)=\frac{i}{\lambda+i\epsilon}\int_0^\infty \beta(\lambda 2^{-j}t) \rho(\epsilon t) e^{i\lambda t} e^{-\epsilon t} \cos (t\tau)\:dt,\quad 1\leq j\leq \lfloor \log_2 (\lambda/\epsilon) \rfloor, \\
    & R_\lambda (\tau)=\frac{i}{\lambda+i\epsilon} \int_0^\infty (1-\rho(\epsilon t)) e^{i\lambda t} e^{-\epsilon t} \cos (t\tau)\:dt.
\end{align*}
Using this, we decompose
\begin{align}\label{resolvent-decomp}
    (-\Delta_g-(\lambda+i\epsilon)^2)^{-1}\psi_1(P/\lambda)=\left(T_0(P)+\sum_{1<2^j\leq \lambda} T_j (P)+\sum_{\lambda<2^j\leq \lambda/\epsilon} T_j (P)+R_\lambda (P) \right)\psi_1(P/\lambda).
\end{align}
We consider the terms on the right hand side separately. If $\tau\geq 0$, then, for $N\in \mathbb{N}$,
\begin{align*}
    & |R_\lambda (\tau) \psi_1(\tau/\lambda)|\leq C_N \lambda^{-1} \epsilon^{-1} (1+\epsilon^{-1}|\lambda-\tau|)^{-N}, \\
    & |T_j (\tau)\psi_1(\tau/\lambda)|\leq C_N\lambda^{-2} 2^j (1+\lambda^{-1}2^j|\lambda-\tau|)^{-N},\quad \text{for } 1\leq 2^j\leq \lambda/\epsilon.
\end{align*}
By this, \eqref{log-window-est}, \eqref{setup-psi-1}, and duality,
\begin{align*}
    &\|R_\lambda (P)\psi_1(P/\lambda)f\|_{L^q (M)} \\
    &\leq \sum_{k=0}^\infty \|\mathds{1}_{[k\epsilon, (k+1)\epsilon]}(P)\circ R_\lambda(P)\psi_1(P/\lambda)\circ \mathds{1}_{[k\epsilon, (k+1)\epsilon]} (P)f\|_{L^q (M)} \\
    &\lesssim \sum_{0\leq k\leq 4\lambda\epsilon^{-1}} ((k+1)\epsilon)^{\sigma(q)} \epsilon^{\frac{1}{2}} \|R_\lambda (P)\psi_1(P/\lambda)\circ \mathds{1}_{[k\epsilon, (k+1)\epsilon]} (P)f\|_{L^2 (M)} \\
    &\leq \sum_{0\leq k\leq 4\lambda\epsilon^{-1}} ((k+1)\epsilon)^{\sigma(q)} \epsilon^{\frac{1}{2}} \left(\sup_{\tau\in [k\epsilon, (k+1)\epsilon]} R_\lambda (\tau)\psi_1(\tau/\lambda) \right)\|\mathds{1}_{[k\epsilon, (k+1)\epsilon]} (P)f\|_{L^2 (M)} \\
    &\lesssim \sum_{0\leq k\leq 4\lambda\epsilon^{-1}} ((k+1)\epsilon)^{\sigma(q)} \epsilon^{\frac{1}{2}} \cdot \lambda^{-1}\epsilon^{-1}\left(\sup_{\tau\in [k\epsilon, (k+1)\epsilon]} (1+\epsilon^{-1} |\lambda-\tau|)^{-N} \right) \cdot ((k+1)\epsilon)^{\sigma\left(\frac{2n}{n-2}\right)}\epsilon^{\frac{1}{2}}\|f\|_{L^{\frac{2n}{n+2}}(M)} \\
    &\lesssim \lambda^{\sigma(q)-\frac{1}{2}}\sum_{0\leq k\leq 4\lambda\epsilon^{-1}} \sup_{\tau\in [k\epsilon, (k+1)\epsilon]} (1+\epsilon^{-1} |\lambda-\tau|)^{-N} \|f\|_{L^{\frac{2n}{n+2}}(M)} \\
    &\lesssim \lambda^{\sigma(q)-\frac{1}{2}} \|f\|_{L^{\frac{2n}{n+2}}(M)},
\end{align*}
which satisfies \eqref{ineq:prop-(16)} without a $(\log \lambda)^{\epsilon'}$-loss, resulting in no loss of $(\log \lambda)^{\epsilon'}$ in \eqref{ineqs:result-(16)}. Similarly, one can compute
\begin{align}\label{ineq:prop-(18)}
    \|T_j (P)\psi_1(P/\lambda)f\|_{L^q (M)}\lesssim \lambda^{\sigma(q)-\frac{1}{2}} \|f\|_{L^{\frac{2n}{n+2}}(M)},\quad 0\leq j\leq \lambda/\epsilon,
\end{align}
and so,
\begin{align}\label{ineq:Tj-log(log)-loss}
    \begin{split}
        \sum_{\lambda<2^j\leq \lambda/\epsilon} \|T_j (P)\psi_1(P/\lambda) f\|_{L^q (M)}\lesssim \lambda^{\sigma(q)-\frac{1}{2}} \log (\log \lambda) \|f\|_{L^{\frac{2n}{n+2}}(M)}.
    \end{split}
\end{align}
This satisfies \eqref{ineq:prop-(16)} for $n\geq 3$, and this estimate \eqref{ineq:Tj-log(log)-loss} is the reason why we have a $(\log \lambda)^{\epsilon'}$-loss in some cases for $n\geq 3$ and $V\in \mathcal{K}(M)\cap L^{\frac{n}{2}} (M)$ in \eqref{setup-kappa}.

If $n=2$, we shall use an estimate other than \eqref{ineq:prop-(16)} to obtain a better estimate than \eqref{ineq:Tj-log(log)-loss}, since we want to remove the $(\log \lambda)^{\epsilon'}$-loss for $2<q<6$ when $n=2$. When $n=2$, if $M$ has negative sectional curvatures, by (the proof of) \cite[(3.25)]{BlairHuangSireSogge2022UniformSobolev} (see also \cite[(4.7)]{HuangWangZhang2026restriction}) we have
\begin{align*}
    \left|\sum_{\lambda<2^j<\lambda/\epsilon} T_j (P) \psi_1(P/\lambda)(x, y) \right|\leq C_{\delta_0} \lambda^{-\frac{1}{2}+\delta_0},\quad \text{for any } \delta_0>0.
\end{align*}
We note that $\frac{2n}{n+2}=1$ when $n=2$. By this and Young's inequality, we have that, for $q>2$,
\begin{align*}
    \sum_{\lambda<2^j\leq \lambda/\epsilon} \|T_j (P)\psi_1(P/\lambda)f\|_{L^q (M)}\lesssim \lambda^{-\frac{1}{2}+\delta_0}\|f\|_{L^1 (M)}\leq \lambda^{\sigma(q)-\frac{1}{2}} \|f\|_{L^{\frac{2n}{n+2}}(M)},
\end{align*}
which still satisfies the bound in \eqref{ineq:prop-(16)} without the loss $(\log \lambda)^{\epsilon'}$. In the last inequality, we used the fact that $\sigma(q)>0$ for $q>2$ and for any such fixed $q$, we can choose $\delta_0>0$ small enough so that $\delta_0<\sigma(q)$.

Note that the case $j=0$ is already considered in \eqref{ineq:prop-(18)}, so we are left to consider the case $1<2^j\leq \lambda$ for $T_j (P) \psi_1(P/\lambda)$ in \eqref{resolvent-decomp} to satisfy the bound in \eqref{ineqs:result-(16)}. By the proof of \cite[(2.23)]{ShaoYao2014UniformSobolev} (see also \cite[(4.8)]{HuangWangZhang2026restriction}), we can write
\begin{align*}
    T_j (P)\psi_1(P/\lambda)(x, y)=S_j (x, y)+W_j (x, y),
\end{align*}
where
\begin{align*}
    & S_j (x, y)=\lambda^{\frac{n-3}{2}}(\lambda^{-1}2^j)^{-\frac{n-1}{2}} e^{i\lambda d_g (x, y)} a_j (x, y), \\
    & W_j (x, y)=O(\lambda^{-1}(\lambda^{-1}2^j)^{1-n}\mathds{1}_{d_g (x, y)<4(\lambda^{-1}2^j)}(x, y)),
\end{align*}
and $a_j\in C_0^\infty$ satisfies
\begin{align}\label{aj-property}
    \begin{split}
        & \mathrm{supp}(a_j)\subset \left\{d_g (x, y)\in \left( \lambda^{-1} 2^{j-2}, \lambda^{-1}2^{j+2} \right)\right\}, \\
        & |\partial_{s, y}^\alpha a_j (x, y)|\leq C_\alpha (\lambda^{-1}2^j)^{-|\alpha|},\quad \text{for all multi-indices } \alpha.
    \end{split}
\end{align}
For $W_j$, we apply Young's inequality. We note that, for $\frac{1}{r}=1-\left(\frac{n+2}{2n}-\frac{1}{q} \right)$,
\begin{align*}
    \left(\int |W_j (x, y)|^r\:dx \right)^{\frac{1}{r}}, \left(\int |W_j (x, y)|^r\:dy \right)^{\frac{1}{r}} &\lesssim \lambda^{-1} (\lambda^{-1} 2^j)^{1-n} \left(\int_{|z|\lesssim \lambda^{-1}2^j, z\in \mathbb{R}^n} 1\cdot dz \right)^{\frac{1}{r}} \\
    &\lesssim \lambda^{\frac{n-2}{2}-\frac{n}{q}}(2^j)^{\frac{n}{q}-\frac{n}{2}}.
\end{align*}
Since $q>2$, summing over all $j$, we obtain $\sum_j\|W_j\|_{L^q(M)\to L^2 (M)}\lesssim \lambda^{\frac{n-2}{2}-\frac{n}{q}}\leq \lambda^{\sigma(q)-\frac{1}{2}}$, which satisfies \eqref{ineq:prop-(16)} without a $(\log \lambda)^{\epsilon'}$-loss.

For $S_j$, we shall find estimates other than \eqref{ineq:prop-(16)} by using scaling arguments as usual. When we consider $S_j f(x)$, since the kernel $S_j (x, y)$ vanishes when $d_g (x, y)\not\in [\lambda^{-1}2^{j-2}, \lambda^{-1}2^{j+2}]$, using a partition of unity if necessary, we may assume that
\begin{align}\label{cond:supp-f}
    \mathrm{supp}(f)\subset B_{\lambda^{-1}2^j}(0),
\end{align}
where $B_{\lambda^{-1}2^j}(0)$ denotes the ball of radius $\lambda^{-1}2^j$ centered at the origin. If we set
\begin{align*}
    & x=\lambda^{-1}2^j X,\quad y=\lambda^{-1}2^j Y,\quad f_j (Y)=f(\lambda^{-1}2^j Y),\\
    & d_j (X, Y)=(\lambda 2^{-j})d_g (\lambda^{-1}2^j X, \lambda^{-1}2^j Y),\quad \tilde{a}_j (X, Y)=a_j (\lambda^{-1}2^j X, \lambda^{-1}2^j Y),
\end{align*}
then we can write
\begin{align*}
    S_j f(x)&=\lambda^{\frac{n-3}{2}}(\lambda^{-1}2^j)^{-\frac{n-1}{2}}\int e^{i\lambda d_g (x, y)} a_j (x, y)f(y)\: dy \\
    &=\lambda^{\frac{n-3}{2}}(\lambda^{-1}2^j)^{-\frac{n-1}{2}}\int e^{i\lambda d_g (\lambda^{-1}2^jX, \lambda^{-1}2^j Y)} a_j (\lambda^{-1}2^j X, \lambda^{-1}2^j Y)f(\lambda^{-1}2^j Y) (\lambda^{-1}2^j)^n\:dY \\
    &=:\lambda^{\frac{n-3}{2}}(\lambda^{-1}2^j)^{\frac{n+1}{2}}\tilde{S}_j f_j (X),
\end{align*}
where
\begin{align*}
    \tilde{S}_j f_j (X)=\int e^{i 2^j d_j (X, Y)} \tilde{a}_j (X, Y)f_j (Y)\:dY.
\end{align*}
By \eqref{cond:supp-f}, we may assume
\begin{align*}
    \mathrm{supp}(f_j)\subset B_1 (0).
\end{align*}
We also note that by the size estimates in \eqref{aj-property} and the scalings, we have
\begin{align*}
    |\partial_{X, Y}^\alpha \tilde{a}_j (X, Y)|\leq C_\alpha,\quad \text{for all } \alpha.
\end{align*}
Moreover, we note that $d_j (X, Y)$ is the Riemannian distance function between $X$ and $Y$ with a stretch metric $g_{ij}(\lambda^{-1}2^j X)$, that is, $d_j (X, Y)$ satisfies the $n\times n$ Carleson-Sj\"olin condition (cf. \cite[Lemma 5.1.3]{Sogge2017FourierBook}). With this in mind, by \cite{Sogge1988Fourier} (see also the proof of \cite[Lemma 5.1.3]{Sogge2017FourierBook}), that
\begin{align}\label{ineq:Sogge-scaling}
    \begin{split}
        \|\tilde{S}_j f_j\|_{L_X^q (M)}\lesssim (2^j)^{-\frac{n-1}{2}+\sigma(q)}\|f_j\|_{L_Y^2 (M)},\quad 2<q\leq \infty.
    \end{split}
\end{align}
On the other hand, since $|\tilde{S}_j (X, Y)|\lesssim 1$, by Young's inequality, we also have a trivial $L_Y^p (M)\to L_X^q (M)$ estimate
\begin{align}\label{ineq:trivial-scaling}
    \|\tilde{S}_j f_j\|_{L_X^q (M)}\lesssim \|f_j\|_{L_Y^p (M)},\quad 1\leq p\leq q\leq \infty.
\end{align}
We are now ready to prove Case 2-2-1, and Case 2-2-2.

\subsubsection{\texorpdfstring{Case 2-2-1: $n=2$, $V\in \mathcal{K}(M)$, and $2<q<6$}{Case 2-2-1}}\label{subsec-Case-2-2-1}
As above, we note that $\frac{2n}{n+2}=1$ when $n=2$. By \eqref{ineq:trivial-scaling}, when $n=2$,
\begin{align*}
    \|\tilde{S}_j f_j\|_{L_X^q (M)}\lesssim \|f_j\|_{L_Y^1 (M)}.
\end{align*}
Unpacking the definition of $S_j$, we have
\begin{align*}
    \|S_j f\|_{L_x^q (M)}&=\lambda^{-\frac{1}{2}}(\lambda^{-1}2^j)^{\frac{3}{2}+\frac{2}{q}}\|\tilde{S}_jf_j\|_{L_Y^q (M)} \\
    &\lesssim \lambda^{-\frac{1}{2}}(\lambda^{-1}2^j)^{\frac{3}{2}+\frac{2}{q}}\|f_j\|_{L^1 (M)} \\
    &=\lambda^{-\frac{1}{2}}(\lambda^{-1}2^j)^{-\frac{1}{2}+\frac{2}{q}}\|f\|_{L^1 (M)} \\
    &=\lambda^{-\frac{2}{q}}(2^j)^{-\frac{1}{2}+\frac{2}{q}}\|f\|_{L^1 (M)}.
\end{align*}
Thus, we would have \eqref{ineq:prop-(16)}, if we could show that
\begin{align*}
    \sum_{1<2^j\leq \lambda} \lambda^{-\frac{2}{q}}(2^j)^{-\frac{1}{2}+\frac{2}{q}}\leq \lambda^{\sigma(q)-\frac{1}{2}},\quad 2<q<6.
\end{align*}
By a direct computation, if $4<q<6$, then
\begin{align*}
    \sum_{1<2^j\leq \lambda} \lambda^{-\frac{2}{q}}(2^j)^{-\frac{1}{2}+\frac{2}{q}}\lesssim \lambda^{-\frac{2}{q}}=\lambda^{\left(\frac{1}{2}-\frac{2}{q}\right)-\frac{1}{2}}\leq \lambda^{\frac{1}{2}\left(\frac{1}{2}-\frac{1}{q}\right)-\frac{1}{2}}=\lambda^{\sigma(q)-\frac{1}{2}}.
\end{align*}
If $q=4$, then, since $\sigma(4)=\frac{1}{8}$, we have
\begin{align*}
    \sum_{1<2^j\leq \lambda} \lambda^{-\frac{2}{q}}(2^j)^{-\frac{1}{2}+\frac{2}{q}}=\sum_{1<2^j\leq \lambda} \lambda^{-\frac{1}{2}}\lesssim \lambda^{-\frac{1}{2}}\log\lambda\lesssim \lambda^{\frac{1}{8}-\frac{1}{2}}=\lambda^{\sigma(4)-\frac{1}{2}}.
\end{align*}
If $2<q<4$, then
\begin{align*}
    \sum_{1<2^j\leq \lambda} \lambda^{-\frac{2}{q}}(2^j)^{-\frac{1}{2}+\frac{2}{q}}\lesssim \lambda^{-\frac{1}{2}}<\lambda^{\frac{1}{2}\left(\frac{1}{2}-\frac{1}{q}\right)-\frac{1}{2}}=\lambda^{\sigma(q)-\frac{1}{2}}.
\end{align*}
Hence, \eqref{ineq:prop-(16)} without a $(\log \lambda)^{\epsilon'}$-loss holds for Case 2-2-1, resulting in \eqref{ineqs:result-(16)} without a $(\log \lambda)^{\epsilon'}$-loss, as desired.

\subsubsection{\texorpdfstring{Case 2-2-2: $n\geq 3$, $V\in \mathcal{K}(M)\cap L^{\frac{n}{2}}(M)$, and $\frac{2n}{n-1}<q<\frac{2n}{n-3}$}{Case 2-2-2}}\label{sec-case-2-2-2}
In Case 2-2-2, we need some variants of \eqref{ineq:prop-(16)} to conclude Case 2-2-2. In fact, the variants may not be too much different from \eqref{ineq:prop-(16)} in the sense that the variants shall be on the ``uniform Sobolev line'', i.e., $\frac{n}{p}-\frac{n}{q}=2$. For simplicity, we consider $n\geq 4$, but similar arguments work well for $n=3$ as well.

We start by noting that the adjoint of the operator $\tilde{S}_j$ can be written as
\begin{align*}
    \tilde{S}_j^* g (Y)=\int_M e^{-i2^j d_j (X,Y)}\overline{\tilde{a}_j (X, Y)}g(X)\:dX.
\end{align*}
We know by \cite[Lemma 5.1.3]{Sogge2017FourierBook} that the phase function $(X, Y)\mapsto d_j (X, Y)$ satisfies the $n\times n$ Carleson-Sj\"olin condition, so does the phase function $(Y, X)\mapsto -d_j (X, Y)$, which is the phase function of $\tilde{S}_j^*$. Thus, by the proof of \eqref{ineq:Sogge-scaling},
\begin{align*}
    \|\tilde{S}_j^* f_j \|_{L_X^q (M)}\lesssim (2^j)^{-\frac{n-1}{2}+\sigma(q)}\|f_j\|_{L_X^2 (M)},\quad 2<q\leq \infty.
\end{align*}
Setting $q=\frac{2(n+1)}{n-1}$ here, we have
\begin{align*}
    \|\tilde{S}_j^* f_j \|_{L^{\frac{2(n+1)}{n-1}}(M)}\lesssim (2^j)^{-\frac{n(n-1)}{2(n+1)}}\|f_j\|_{L^2 (M)}.
\end{align*}
By this and duality,
\begin{align}\label{ineq:Sj-tilde-2(n+1)/(n+3)}
    \|\tilde{S}_j f_j \|_{L^2 (M)}\lesssim (2^j)^{-\frac{n(n-1)}{2(n+1)}}\|f_j\|_{L^{\frac{2(n+1)}{n+3}}(M)}.
\end{align}
By \eqref{ineq:trivial-scaling}, we have a trivial $L^1 (M)\to L^\infty (M)$ bound
\begin{align}\label{ineq:Sj-tilde-trivial}
    \|\tilde{S}_j f_j \|_{L^\infty (M)}\lesssim \|f_j\|_{L^1 (M)}.
\end{align}
Interpolating \eqref{ineq:Sj-tilde-2(n+1)/(n+3)} and \eqref{ineq:Sj-tilde-trivial} yields
\begin{align}\label{p0-q0-scaled-est}
    \|\tilde{S}_j f_j\|_{L^{\frac{2n^2}{n^2-n-2}}(M)}\lesssim (2^j)^{-\frac{(n-1)(n-2)}{2n}}\|f_j\|_{L^{\frac{2n^2}{n^2+3n-2}}(M)}.
\end{align}
Here, we note that the pair $(p, q)=\left(\frac{2n^2}{n^2+3n-2}, \frac{2n^2}{n^2-n-2}\right)$ satisfies $\frac{n}{p}-\frac{n}{q}=2$. Also by \eqref{ineq:trivial-scaling}, we have
\begin{align*}
    \|\tilde{S}_j f_j\|_{L^{\frac{n}{n-2}}(M)}\lesssim \|f_j\|_{L^1 (M)}.
\end{align*}
Interpolating this with \eqref{p0-q0-scaled-est}, we have
\begin{align}\label{2n/(n-1)-on-Sobolev}
    \|\tilde{S}_j f_j\|_{L^{\frac{2n}{n-1}}(M)}\lesssim (2^j)^{-\frac{n-3}{2}}\|f_j\|_{\frac{2n}{n+3}(M)}.
\end{align}
We also note that the pair $(p, q)=\left(\frac{2n}{n+3}, \frac{2n}{n-1} \right)$ satisfies $\frac{n}{p}-\frac{n}{q}=2$. Again, interpolating \eqref{2n/(n-1)-on-Sobolev} with \eqref{p0-q0-scaled-est}, we have, for some $\alpha(p, q)>0$,
\begin{align}\label{Sj-tilde-on-Sobolev}
    \|\tilde{S}_j f_j\|_{L^q (M)}\lesssim (2^j)^{-\frac{n-3}{2}-\alpha(p, q)}\|f_j\|_{L^p (M)},\quad \text{where } \frac{n}{p}-\frac{n}{q}=2 \text{ and } \frac{2n}{n-1}<q\leq \frac{2n^2}{n^2-n-2},
\end{align}
since $-\frac{(n-1)(n-2)}{2n}<-\frac{n-3}{2}$. By a direct computation, one can check that
\begin{align}\label{Sj-Sj-tilde}
    \|S_j\|_{L_y^p(M)\to L_x^q(M)}\lesssim (2^j)^{\frac{n-3}{2}}\|\tilde{S}_j\|_{L_Y^p (M)\to L_X^q (M)},\quad \text{where } \frac{n}{p}-\frac{n}{q}=2 \text{ and } 2\leq q\leq \infty.
\end{align}
By \eqref{Sj-tilde-on-Sobolev} and \eqref{Sj-Sj-tilde}, we have, for some $\alpha(p, q)>0$,
\begin{align*}
    \|S_j\|_{L^p (M)\to L^q (M)}\lesssim (2^j)^{-\alpha(p, q)},\quad \text{where } \frac{n}{p}-\frac{n}{q}=2 \text{ and } \frac{2n}{n-1}<q\leq \frac{2n^2}{n^2-n-2}.
\end{align*}
Summing over all $1\leq j\leq \lfloor \log_2 \lambda \rfloor$, we have a uniform estimate $\sum_j \|S_j\|_{L^{p_0}(M)\to L^{q_0}(M)}\lesssim 1$. By this, if we set $S_\lambda=\sum_{1<2^j\leq \lambda} S_j$, then
\begin{align}\label{S-lambda-Sobolev-low-q}
    \|S_\lambda f\|_{L^q (M)}\lesssim \|f\|_{L^p (M)},\quad \text{where } \frac{n}{p}-\frac{n}{q}=2 \text{ and } \frac{2n}{n-1}<q\leq \frac{2n^2}{n^2-n-2}.
\end{align}

To deal with the other exponents $\frac{2n^2}{n^2-n-2}\leq q<\frac{2n}{n-3}$ (this may be a place where we need the assumption $n\geq 4$ to make sure that $\frac{2n}{n-3}$ is well-defined, but our arguments shall work for $n=3$, since ultimately we need \eqref{S-lambda-low-q-est} below for $\frac{2n}{n-1}<q<\frac{2(n+1)}{n-1}$, so let us focus on $n\geq 4$ here, for simplicity), as before, one can obtain
\begin{align*}
    \|\tilde{S}_j f_j\|_{L^{\frac{2(n+1)}{n-1}}(M)}&\lesssim (2^j)^{-\frac{n(n-1)}{2(n+1)}}\|f_j\|_{L^2 (M)}, \\
    \|\tilde{S}_j f_j\|_{L^\infty (M)}&\lesssim \|f_j\|_{L^1 (M)}.
\end{align*}
By interpolation,
\begin{align*}
    \|\tilde{S}_j f_j\|_{L^{\frac{2n^2}{(n-1)(n-2)}}(M)}\lesssim (2^j)^{-\frac{(n-1)(n-2)}{2n}}\|f_j\|_{L^{\frac{2n^2}{n^2+n+2}} (M)}.
\end{align*}
Interpolating this with
\begin{align*}
    \|\tilde{S}_j f_j\|_{L^\infty (M)}\lesssim \|f_j\|_{L^{\frac{n}{2}}(M)},
\end{align*}
we have
\begin{align*}
    \|\tilde{S}_j f_j\|_{L^{\frac{2n}{n-3}}(M)}\lesssim (2^j)^{-\frac{n-3}{2}}\|f_j\|_{L^{\frac{2n}{n+1}}(M)}.
\end{align*}
Again, interpolating this with \eqref{p0-q0-scaled-est}, we have, for some $\tilde{\alpha}(p, q)>0$,
\begin{align*}
    \|\tilde{S}_j f_j\|_{L^q (M)}\lesssim (2^j)^{-\frac{n-3}{2}-\tilde{\alpha}(p, q)}\|f_j\|_{L^p (M)},\quad \frac{n}{p}-\frac{n}{q}=2,\quad \frac{2n^2}{n^2-n-2}\leq q<\frac{2n}{n-3}.
\end{align*}
Summing over all $1\leq j\leq \lfloor \log_2 \lambda \rfloor$, we have
\begin{align*}
    \|S_\lambda f\|_{L^q (M)}\lesssim \|f\|_{L^p (M)},\quad \text{where } \frac{n}{p}-\frac{n}{q}=2 \text{ and } \frac{2n^2}{n^2-n-2}\leq q<\frac{2n}{n-3}.
\end{align*}
By this and \eqref{S-lambda-Sobolev-low-q}, we have
\begin{align}\label{S-lambda-unif-Sobolev}
    \|S_\lambda f\|_{L^q (M)}\lesssim \|f\|_{L^p (M)},\quad \text{where } \frac{n}{p}-\frac{n}{q}=2 \text{ and } \frac{2n}{n-1}<q<\frac{2n}{n-3}.
\end{align}
We also note that if $\frac{n}{p}-\frac{n}{q}=2$, then $\frac{np}{n-2p}=q$. With this in mind, by \eqref{S-lambda-unif-Sobolev} and H\"older's inequality, if $\frac{n}{p}-\frac{n}{q}=2$ and $\frac{2n}{n-1}<q<\frac{2n}{n-3}$, then we have
\begin{align}\label{S-lambda-low-q-est}
    \begin{split}
        \|S_\lambda (V\cdot \mathds{1}_\lambda (P_V)f)\|_{L^q (M)}&\lesssim \|V\cdot \mathds{1}_\lambda (P_V)f\|_{L^p (M)} \\
        &\leq \|V\|_{L^{\frac{n}{2}}(M)}\|\mathds{1}_\lambda(P_V)f\|_{L^{\frac{np}{n-2p}}(M)} \\
        &\leq C_V \frac{\lambda^{\sigma(q)}}{(\log \lambda)^{\frac{1}{2}}}\|f\|_{L^2 (M)}, \quad \frac{2n}{n-1}<q\leq \frac{2n}{n-3},
    \end{split}
\end{align}
which shows that the bound of \eqref{ineqs:result-(16)}, without a $(\log \lambda)^{\epsilon'}$-loss, still holds for $\frac{2n}{n-1}<q<\frac{2n}{n-3}$, even though technically we did not check if \eqref{ineq:prop-(16)} holds without a $(\log \lambda)^{\epsilon'}$-loss. This completes the proof of Case 2-2-3, and hence, Proposition \ref{prop:hv-Kato-and-Ln2}.

\section{Future directions}\label{sec:future}

The results of this paper raise a number of natural questions and suggest several avenues for further development. We outline below what we believe are the most pressing directions, ranging from sharpness and optimality to extensions involving non-self-adjoint perturbations and critical potential classes.

\subsection{Sharpness and optimality}

A fundamental question left open by Theorems \ref{thm:UP-FR-on}, \ref{thm:hv-Kato-or-Ln2}, and \ref{thm:UP-nonposit-curv} concerns the sharpness of the exponents and logarithmic gains obtained. Specifically, one may ask whether the exponents $\sigma(q)$ appearing in \eqref{ineq:UP-FR-on} and the logarithmic exponents $\delta_n(q)$ and $\kappa(q)$ in \eqref{ineq:V=0-nonpositive} and \eqref{ineq:s2-all} are optimal. On manifolds with nonpositive or negative sectional curvature, do there exist sequences of spectrally localized functions $f_\lambda$ supported on sets $E_\lambda$ such that the inequalities \eqref{ineq:V=0-nonpositive} become equalities up to constants? Constructing such extremal functions—likely using distorted plane waves or approximate eigenfunctions concentrating near closed geodesics—would demonstrate the optimality of our curvature-dependent improvements and is a natural next step.

\subsection{Stability and quantitative recovery}

The uniqueness results for recovery from incomplete spectral data established in Section \ref{sec:recovery} are purely qualitative. In practical applications, one rarely has exact agreement on the observed spectral coefficients; rather, one has approximate data. This motivates the study of \emph{stability}: if two functions agree approximately on the observed spectral data, can we bound the distance between them in $L^2(M)$? A natural conjecture is that the uncertainty principles proved above imply a Donoho--Stark type stability estimate of the form
\[
\|f-g\|_{L^2(M)} \leq C\Bigl(\|\Pi_I(f-g)\|_{L^2(M)} + \eta^{-\alpha}\|\Pi_{I^c}(f-g)\|_{L^2(M)}\Bigr)
\]
under appropriate support size conditions, where $\Pi_I$ denotes projection onto the missing spectral window and $\Pi_{I^c}$ onto the observed frequencies. Developing such quantitative estimates would substantially strengthen the applicability of our recovery framework.

\subsection{Manifolds with boundary}

Throughout this paper, we have assumed that $M$ is a compact Riemannian manifold without boundary. An immediate and nontrivial extension is to the setting of compact manifolds with boundary, equipped with, say, Dirichlet or Neumann boundary conditions. Spectral projection estimates for the Laplace-Beltrami operator on manifolds with boundary are more delicate due to the presence of glancing and diffractive phenomena (see, e.g., the work of Smith--Sogge \cite{SmithSogge2007Boundary} and Grieser \cite{Grieser1992PhD-Thesis}). However, recent progress suggests that uniform Sobolev estimates similar to those we have employed may still hold, possibly with additional contributions from the boundary. Extending Theorems \ref{thm:UP-FR-on} and \ref{thm:hv-Kato-or-Ln2} to this setting would require a careful analysis of how boundary conditions interact with the restriction-theoretic approach central to our arguments.

\subsection{Quantum chaos and scarring}

On manifolds with chaotic geodesic flow, the Quantum Unique Ergodicity conjecture (studied by Lindenstrauss \cite{Lindenstrauss2006Ergodicity} and Soundararajan \cite{Soundararajan2010Ergodicity} in the arithmetic setting and by Dyatlov--Jin--Nonnenmacher \cite{DyatlovJinNonnenmacher2022Control} in certain non-arithmetic cases) asserts that eigenfunctions become equidistributed in the high-energy limit. However, exceptional eigenfunctions known as \emph{quantum scars} can exhibit anomalous concentration near unstable periodic geodesics. Our framework provides a tool for quantifying such scarring: if a spectrally localized function $f = \mathds{1}_{[\lambda,\lambda+1]}(\sqrt{-\Delta_g})f$ concentrates near a closed geodesic $\gamma$ (so that its support is contained in a small tubular neighborhood of $\gamma$), then Theorem \ref{thm:UP-FR-on} forces a lower bound on the measure of that neighborhood in terms of $\#S$ and the Fourier ratio $\FR_{2\alpha(q)}(f)$. Investigating the extent to which these bounds are compatible with existing constructions of scarred eigenfunctions would be a fruitful direction at the interface of spectral geometry and quantum chaos.

\subsection{Critical and super-critical potentials}

The class $V\in L^{n/2}(M)$ considered in Theorem \ref{thm:hv-Kato-or-Ln2} is scaling-critical for the Schrödinger operator. Our estimates for $V\in L^{n/2}(M)$ are restricted to the range $2<q\leq 2n/(n-4)$ when $n\geq 5$ and $2<q<\infty$ when $n=3,4$. The endpoint $q=2n/(n-4)$ for $n>4$ is critical, and one may ask whether the estimate extends to $q>2n/(n-4)$ or whether counterexamples exist. This question is intimately connected to the optimality of the uniform Sobolev estimates of Blair--Huang--Sire--Sogge that we rely upon. Moreover, the case $V\in \mathcal{K}(M)\cap L^{n/2}(M)$ in Theorem \ref{thm:UP-nonposit-curv} yields logarithmic improvements for certain exponent ranges, but the presence of the arbitrarily small $\epsilon'$ in \eqref{setup-kappa} for some regimes suggests that the optimal logarithmic exponents may be larger than those we have obtained. Resolving these endpoint issues would require a refined analysis of the spectral projection bounds for singular potentials.

\subsection{Non-self-adjoint perturbations}

A more speculative but potentially far-reaching direction concerns Schrödinger operators with complex-valued potentials, i.e., $H_V = -\Delta_g + V$ where $V$ is complex and not necessarily self-adjoint. Such operators arise in the study of open quantum systems, resonances, and non-Hermitian physics. The spectrum is no longer real, and spectral projections onto spectral windows are more delicate due to the possibility of pseudospectral phenomena. Nevertheless, resolvent estimates for such operators in existing literature might serve as a substitute for the spectral projection bounds used in this paper. One could envision developing a theory of \emph{pseudospectral uncertainty principles}, where the notion of spectral localization is replaced by pseudospectral concentration, with the goal of obtaining analogous support-size lower bounds. Such results would represent a significant departure from the self-adjoint setting and would open new connections to non-Hermitian harmonic analysis.

\subsection{Higher-order and fractional operators}

Finally, we note that the restriction-theoretic approach underlying Theorem \ref{thm:UP-FR-on} is not intrinsically tied to the Laplace-Beltrami operator. One could consider higher-order elliptic operators, such as the poly-Laplacian $(-\Delta_g)^m$ for $m>1$, or fractional powers $(-\Delta_g)^s$ for $0<s<1$. The spectral projection estimates for such operators are less developed, particularly on compact manifolds, but existing work on $L^p$ bounds for spectral clusters of fractional Laplacians may suggest that analogues of Theorems \ref{thm:UP-FR-on} and \ref{thm:UP-nonposit-curv} should hold, possibly with modified exponents reflecting the order of the operator. Developing a unified framework for uncertainty principles for a broad class of pseudodifferential operators on manifolds remains an ambitious but worthwhile goal.

\subsection*{Concluding perspective}

The framework developed in this paper suggests that spectral localization, quantified by the Fourier ratio $\FR_q(f)$ or its analogues, serves as a universal bridge connecting spectral theory, harmonic analysis, and geometric PDE. The results presented here—from the basic restriction-theoretic uncertainty principle on compact manifolds to the refined logarithmic estimates under curvature assumptions and the stability under singular perturbations—indicate that this perspective is both robust and flexible. We hope that the questions raised above will stimulate further research at the intersection of these fields, and that the techniques introduced herein will find applications beyond the specific problems considered in this work.

\bibliographystyle{amsplain}
\bibliography{references}

@book {Sogge2017FourierBook,
    AUTHOR = {Sogge, Christopher D.},
     TITLE = {Fourier integrals in classical analysis},
    SERIES = {Cambridge Tracts in Mathematics},
    VOLUME = {210},
   EDITION = {Second},
 PUBLISHER = {Cambridge University Press, Cambridge},
      YEAR = {2017},
     PAGES = {xiv+334},
      ISBN = {978-1-107-12007-5},
   MRCLASS = {42-02 (35S05 35S30 42B15 42B20 47G30 58J40 58J47)},
  MRNUMBER = {3645429},
       DOI = {10.1017/9781316341186},
       URL = {https://doi.org/10.1017/9781316341186},
}

@article {HassellTacy2015improvement,
    AUTHOR = {Hassell, Andrew and Tacy, Melissa},
     TITLE = {Improvement of eigenfunction estimates on manifolds of
              nonpositive curvature},
   JOURNAL = {Forum Math.},
  FJOURNAL = {Forum Mathematicum},
    VOLUME = {27},
      YEAR = {2015},
    NUMBER = {3},
     PAGES = {1435--1451},
      ISSN = {0933-7741},
   MRCLASS = {35R01 (35P15 53C21 58J50)},
  MRNUMBER = {3341481},
       DOI = {10.1515/forum-2012-0176},
       URL = {https://doi.org/10.1515/forum-2012-0176},
}

@article {BlairSogge2019logarithmic,
    AUTHOR = {Blair, Matthew D. and Sogge, Christopher D.},
     TITLE = {Logarithmic improvements in {$L^p$} bounds for eigenfunctions
              at the critical exponent in the presence of nonpositive
              curvature},
   JOURNAL = {Invent. Math.},
  FJOURNAL = {Inventiones Mathematicae},
    VOLUME = {217},
      YEAR = {2019},
    NUMBER = {2},
     PAGES = {703--748},
      ISSN = {0020-9910},
   MRCLASS = {58J50 (35P15 35R01 58J05 81Q20)},
  MRNUMBER = {3987179},
       DOI = {10.1007/s00222-019-00873-6},
       URL = {https://doi.org/10.1007/s00222-019-00873-6},
}

@article {SmithSogge2007Boundary,
    AUTHOR = {Smith, Hart F. and Sogge, Christopher D.},
     TITLE = {On the {$L^p$} norm of spectral clusters for compact manifolds
              with boundary},
   JOURNAL = {Acta Math.},
  FJOURNAL = {Acta Mathematica},
    VOLUME = {198},
      YEAR = {2007},
    NUMBER = {1},
     PAGES = {107--153},
      ISSN = {0001-5962},
   MRCLASS = {58J50 (35J15 35P15 47F05 47G10)},
  MRNUMBER = {2316270},
MRREVIEWER = {Elena A. Mazepa},
       DOI = {10.1007/s11511-007-0014-z},
       URL = {https://doi-org.libproxy.unm.edu/10.1007/s11511-007-0014-z},
}

@article {BlairSireSogge2021Quasimode,
    AUTHOR = {Blair, Matthew D. and Sire, Yannick and Sogge, Christopher D.},
     TITLE = {Quasimode, eigenfunction and spectral projection bounds for
              {S}chr\"{o}dinger operators on manifolds with critically singular
              potentials},
   JOURNAL = {J. Geom. Anal.},
  FJOURNAL = {Journal of Geometric Analysis},
    VOLUME = {31},
      YEAR = {2021},
    NUMBER = {7},
     PAGES = {6624--6661},
      ISSN = {1050-6926},
   MRCLASS = {35P20 (58J50)},
  MRNUMBER = {4289239},
       DOI = {10.1007/s12220-019-00287-z},
       URL = {https://doi.org/10.1007/s12220-019-00287-z},
}

@article {BlairHuangSireSogge2022UniformSobolev,
    AUTHOR = {Blair, Matthew D. and Huang, Xiaoqi and Sire, Yannick and
              Sogge, Christopher D.},
     TITLE = {Uniform {S}obolev estimates on compact manifolds involving
              singular potentials},
   JOURNAL = {Rev. Mat. Iberoam.},
  FJOURNAL = {Revista Matem\'{a}tica Iberoamericana},
    VOLUME = {38},
      YEAR = {2022},
    NUMBER = {4},
     PAGES = {1239--1286},
      ISSN = {0213-2230,2235-0616},
   MRCLASS = {58J50 (35J10)},
  MRNUMBER = {4445914},
MRREVIEWER = {G.\ V.\ Rozenblum},
       DOI = {10.4171/rmi/1300},
       URL = {https://doi.org/10.4171/rmi/1300},
}

@article {Simon1982Semigroup,
    AUTHOR = {Simon, Barry},
     TITLE = {Schr\"{o}dinger semigroups},
   JOURNAL = {Bull. Amer. Math. Soc. (N.S.)},
  FJOURNAL = {American Mathematical Society. Bulletin. New Series},
    VOLUME = {7},
      YEAR = {1982},
    NUMBER = {3},
     PAGES = {447--526},
      ISSN = {0273-0979},
   MRCLASS = {81-02 (35-02 35P05 47D05 47F05 81C10)},
  MRNUMBER = {670130},
MRREVIEWER = {Ren\'{e} Carmona},
       DOI = {10.1090/S0273-0979-1982-15041-8},
       URL = {https://doi-org.libproxy.unm.edu/10.1090/S0273-0979-1982-15041-8},
}

@article {BourgainShaoSoggeYao2015Resolvent,
    AUTHOR = {Bourgain, Jean and Shao, Peng and Sogge, Christopher D. and
              Yao, Xiaohua},
     TITLE = {On {$L^p$}-resolvent estimates and the density of eigenvalues
              for compact {R}iemannian manifolds},
   JOURNAL = {Comm. Math. Phys.},
  FJOURNAL = {Communications in Mathematical Physics},
    VOLUME = {333},
      YEAR = {2015},
    NUMBER = {3},
     PAGES = {1483--1527},
      ISSN = {0010-3616},
   MRCLASS = {58J50 (47A10)},
  MRNUMBER = {3302640},
MRREVIEWER = {G. V. Rozenblum},
       DOI = {10.1007/s00220-014-2077-y},
       URL = {https://doi.org/10.1007/s00220-014-2077-y},
}

@article {KenigRuizSogge1987UniformSobolev,
    AUTHOR = {Kenig, C. E. and Ruiz, A. and Sogge, C. D.},
     TITLE = {Uniform {S}obolev inequalities and unique continuation for
              second order constant coefficient differential operators},
   JOURNAL = {Duke Math. J.},
  FJOURNAL = {Duke Mathematical Journal},
    VOLUME = {55},
      YEAR = {1987},
    NUMBER = {2},
     PAGES = {329--347},
      ISSN = {0012-7094},
   MRCLASS = {35E99 (35B45 35B60)},
  MRNUMBER = {894584},
MRREVIEWER = {Claude Zuily},
       DOI = {10.1215/S0012-7094-87-05518-9},
       URL = {https://doi-org.libproxy.unm.edu/10.1215/S0012-7094-87-05518-9},
}

@article {DSFKenigSalo2014Forum,
    AUTHOR = {Dos Santos Ferreira, David and Kenig, Carlos E. and Salo,
              Mikko},
     TITLE = {On {$L^p$} resolvent estimates for {L}aplace-{B}eltrami
              operators on compact manifolds},
   JOURNAL = {Forum Math.},
  FJOURNAL = {Forum Mathematicum},
    VOLUME = {26},
      YEAR = {2014},
    NUMBER = {3},
     PAGES = {815--849},
      ISSN = {0933-7741},
   MRCLASS = {35R01 (35J05 35P05 35R30 42B37)},
  MRNUMBER = {3200351},
MRREVIEWER = {Qi Han},
       DOI = {10.1515/forum-2011-0157},
       URL = {https://doi.org/10.1515/forum-2011-0157},
}

@misc{BlairHuangSogge2022Improved,
Author = {Matthew D. Blair and Xiaoqi Huang and Christopher D. Sogge},
Title = {Improved spectral projection estimates},
Year = {2022},
Note = {to appear J. Eur. Math. Soc. (JEMS), published online first, DOI 10.4171/JEMS/1571},
Howpublished = {\url{https://ems.press/journals/jems/articles/14298419}},
Eprint = {DOI 10.4171/JEMS/1571},
}

@article {S.Huang-Sogge2014Resolvent,
    AUTHOR = {Huang, Shanlin and Sogge, Christopher D.},
     TITLE = {Concerning {$L^p$} resolvent estimates for simply connected
              manifolds of constant curvature},
   JOURNAL = {J. Funct. Anal.},
  FJOURNAL = {Journal of Functional Analysis},
    VOLUME = {267},
      YEAR = {2014},
    NUMBER = {12},
     PAGES = {4635--4666},
      ISSN = {0022-1236},
   MRCLASS = {58C40 (35R01)},
  MRNUMBER = {3275105},
MRREVIEWER = {Mohammed El A\"{\i}di},
       DOI = {10.1016/j.jfa.2014.08.016},
       URL = {https://doi.org/10.1016/j.jfa.2014.08.016},
}

@article {Guneysu2012OnGeneralizedSchrodingerSemigroups,
    AUTHOR = {G{\"{u}}neysu, Batu},
     TITLE = {On generalized {S}chr\"odinger semigroups},
   JOURNAL = {J. Funct. Anal.},
  FJOURNAL = {Journal of Functional Analysis},
    VOLUME = {262},
      YEAR = {2012},
    NUMBER = {11},
     PAGES = {4639--4674},
      ISSN = {0022-1236,1096-0783},
   MRCLASS = {58J99 (47D08 53B21 53B50 53C20 81Q35)},
  MRNUMBER = {2913682},
MRREVIEWER = {Isamu\ D\^oku},
       DOI = {10.1016/j.jfa.2011.11.030},
       URL = {https://doi.org/10.1016/j.jfa.2011.11.030},
}

@article {Sturm1993SchrodingerSemigroups,
    AUTHOR = {Sturm, Karl-Theodor},
     TITLE = {Schr\"odinger semigroups on manifolds},
   JOURNAL = {J. Funct. Anal.},
  FJOURNAL = {Journal of Functional Analysis},
    VOLUME = {118},
      YEAR = {1993},
    NUMBER = {2},
     PAGES = {309--350},
      ISSN = {0022-1236,1096-0783},
   MRCLASS = {58G11 (31C12 35J10 47D06 47F05 58G30)},
  MRNUMBER = {1250266},
MRREVIEWER = {Michael\ R\"ockner},
       DOI = {10.1006/jfan.1993.1147},
       URL = {https://doi.org/10.1006/jfan.1993.1147},
}

@article {FrankSabin2017RestrictionTheorem,
    AUTHOR = {Frank, Rupert L. and Sabin, Julien},
     TITLE = {Restriction theorems for orthonormal functions, {S}trichartz
              inequalities, and uniform {S}obolev estimates},
   JOURNAL = {Amer. J. Math.},
  FJOURNAL = {American Journal of Mathematics},
    VOLUME = {139},
      YEAR = {2017},
    NUMBER = {6},
     PAGES = {1649--1691},
      ISSN = {0002-9327,1080-6377},
   MRCLASS = {42B25 (46E30)},
  MRNUMBER = {3730931},
MRREVIEWER = {Oscar\ Blasco},
       DOI = {10.1353/ajm.2017.0041},
       URL = {https://doi.org/10.1353/ajm.2017.0041},
}

@article {BlairPark2025LqEstimates,
    AUTHOR = {Blair, Matthew D. and Park, Chamsol},
     TITLE = {{$L^q$} estimates on the restriction of {S}chr\"odinger
              eigenfunctions with singular potentials},
   JOURNAL = {Comm. Partial Differential Equations},
  FJOURNAL = {Communications in Partial Differential Equations},
    VOLUME = {50},
      YEAR = {2025},
    NUMBER = {10-12},
     PAGES = {1234--1290},
      ISSN = {0360-5302,1532-4133},
   MRCLASS = {58J40 (35J10 35S30 42B37)},
  MRNUMBER = {5000983},
       DOI = {10.1080/03605302.2025.2551500},
       URL = {https://doi.org/10.1080/03605302.2025.2551500},
}

@article {Hickman2020UniformResolventEstimates,
    AUTHOR = {Hickman, Jonathan},
     TITLE = {Uniform {$L^p$} resolvent estimates on the torus},
   JOURNAL = {Math. Res. Rep.},
  FJOURNAL = {Mathematics Research Reports},
    VOLUME = {1},
      YEAR = {2020},
     PAGES = {31--45},
      ISSN = {2772-9559},
   MRCLASS = {35J05 (11P21)},
  MRNUMBER = {4387449},
}

@article {HuangSogge2021Weyl,
    AUTHOR = {Huang, Xiaoqi and Sogge, Christopher D.},
     TITLE = {Weyl formulae for {S}chr\"odinger operators with critically
              singular potentials},
   JOURNAL = {Comm. Partial Differential Equations},
  FJOURNAL = {Communications in Partial Differential Equations},
    VOLUME = {46},
      YEAR = {2021},
    NUMBER = {11},
     PAGES = {2088--2133},
      ISSN = {0360-5302,1532-4133},
   MRCLASS = {58J50 (35P15)},
  MRNUMBER = {4313448},
MRREVIEWER = {He-Jun\ Sun},
       DOI = {10.1080/03605302.2021.1925915},
       URL = {https://doi.org/10.1080/03605302.2021.1925915},
}

@article {HuangZhang2022PtwiseWeyl,
    AUTHOR = {Huang, Xiaoqi and Zhang, Cheng},
     TITLE = {Pointwise {W}eyl laws for {S}chr\"odinger operators with
              singular potentials},
   JOURNAL = {Adv. Math.},
  FJOURNAL = {Advances in Mathematics},
    VOLUME = {410},
      YEAR = {2022},
     PAGES = {Paper No. 108688, 34},
      ISSN = {0001-8708,1090-2082},
   MRCLASS = {58J50 (35P15 81Q10)},
  MRNUMBER = {4491246},
MRREVIEWER = {He-Jun\ Sun},
       DOI = {10.1016/j.aim.2022.108688},
       URL = {https://doi.org/10.1016/j.aim.2022.108688},
}

@article {HuangZhang2023SharpPtwiseWeyl,
    AUTHOR = {Huang, Xiaoqi and Zhang, Cheng},
     TITLE = {Sharp pointwise {W}eyl laws for {S}chr\"odinger operators with
              singular potentials on flat tori},
   JOURNAL = {Comm. Math. Phys.},
  FJOURNAL = {Communications in Mathematical Physics},
    VOLUME = {401},
      YEAR = {2023},
    NUMBER = {2},
     PAGES = {1063--1125},
      ISSN = {0010-3616,1432-0916},
   MRCLASS = {58J50 (35J10)},
  MRNUMBER = {4610271},
MRREVIEWER = {E.\ Pozzoli},
       DOI = {10.1007/s00220-023-04665-1},
       URL = {https://doi.org/10.1007/s00220-023-04665-1},
}

@article {LiYau1986Parabolic,
    AUTHOR = {Li, Peter and Yau, Shing-Tung},
     TITLE = {On the parabolic kernel of the {S}chr\"odinger operator},
   JOURNAL = {Acta Math.},
  FJOURNAL = {Acta Mathematica},
    VOLUME = {156},
      YEAR = {1986},
    NUMBER = {3-4},
     PAGES = {153--201},
      ISSN = {0001-5962,1871-2509},
   MRCLASS = {58G11 (35J10)},
  MRNUMBER = {834612},
MRREVIEWER = {Harold\ Donnelly},
       DOI = {10.1007/BF02399203},
       URL = {https://doi.org/10.1007/BF02399203},
}

@misc{IosevichMayeliWyman2026Spectral,
  Author={Alex Iosevich and Azita Mayeli and Emmett Wyman},
  Title={Spectral synthesis on {R}iemannian manifolds},
  Year={2026},
  Note={Preprint},
  Howpublished = {\url{https://arxiv.org/abs/2603.21451}},
  Eprint = {arxiv:2603.21451},
}

@article {HuangWangZhang2026restriction,
    AUTHOR = {Huang, Xiaoqi and Wang, Xing and Zhang, Cheng},
     TITLE = {Restriction of {S}chr\"odinger {E}igenfunctions to
              {S}ubmanifolds},
   JOURNAL = {Comm. Math. Phys.},
  FJOURNAL = {Communications in Mathematical Physics},
    VOLUME = {407},
      YEAR = {2026},
    NUMBER = {4},
     PAGES = {Paper No. 67},
      ISSN = {0010-3616,1432-0916},
   MRCLASS = {99-06},
  MRNUMBER = {5041854},
       DOI = {10.1007/s00220-026-05576-7},
       URL = {https://doi.org/10.1007/s00220-026-05576-7},
}

@misc{IosevichMayeliWyman2024Uncertainty,
  Author={Alex Iosevich and Azita Mayeli and Emmett Wyman},
  Title={Fourier Uncertainty Principles on {R}iemannian Manifolds},
  Year={2024},
  Howpublished = {\url{https://arxiv.org/abs/2411.09057}},
  Eprint = {arXiv:2411.09057},
}

@article {FrankSabin2017Spectral,
    AUTHOR = {Frank, Rupert L. and Sabin, Julien},
     TITLE = {Spectral cluster bounds for orthonormal systems and
              oscillatory integral operators in {S}chatten spaces},
   JOURNAL = {Adv. Math.},
  FJOURNAL = {Advances in Mathematics},
    VOLUME = {317},
      YEAR = {2017},
     PAGES = {157--192},
      ISSN = {0001-8708,1090-2082},
   MRCLASS = {42B20 (35P05 42B25)},
  MRNUMBER = {3682666},
MRREVIEWER = {Christoph\ Kriegler},
       DOI = {10.1016/j.aim.2017.06.023},
       URL = {https://doi.org/10.1016/j.aim.2017.06.023},
}

@article{Sogge1988Fourier,
  author = {Sogge, Christopher D.},
  title = {Concerning the {$L^p$} norm of spectral clusters for second-order elliptic operators on compact manifolds},
  journal = {J. Funct. Anal.},
  volume = {77},
  number = {1},
  pages = {123--138},
  year = {1988}
}

@article {BlairPark2026Resolvent,
    AUTHOR = {Blair, Matthew D. and Park, Chamsol},
     TITLE = {{$L^p-L^q$} resolvent restriction estimates for submanifolds},
   JOURNAL = {Proc. Amer. Math. Soc.},
  FJOURNAL = {Proceedings of the American Mathematical Society},
    VOLUME = {154},
      YEAR = {2026},
    NUMBER = {6},
     PAGES = {2563--2576},
      ISSN = {0002-9939,1088-6826},
   MRCLASS = {35S30 (42B37 58J40)},
  MRNUMBER = {5065058},
       DOI = {10.1090/proc/17612},
       URL = {https://doi.org/10.1090/proc/17612},
}

@article {HuangSogge2025Curvature,
    AUTHOR = {Huang, Xiaoqi and Sogge, Christopher D.},
     TITLE = {Curvature and sharp growth rates of log-quasimodes on compact
              manifolds},
   JOURNAL = {Invent. Math.},
  FJOURNAL = {Inventiones Mathematicae},
    VOLUME = {239},
      YEAR = {2025},
    NUMBER = {3},
     PAGES = {947--1008},
      ISSN = {0020-9910,1432-1297},
   MRCLASS = {58J50 (35J05 35P15 35R01)},
  MRNUMBER = {4861130},
MRREVIEWER = {Alexander\ G.\ Losev},
       DOI = {10.1007/s00222-025-01315-2},
       URL = {https://doi.org/10.1007/s00222-025-01315-2},
}

@article {CanzaniGalkowski2023APDE,
    AUTHOR = {Canzani, Yaiza and Galkowski, Jeffrey},
     TITLE = {Growth of high {$L^p$} norms for eigenfunctions: an
              application of geodesic beams},
   JOURNAL = {Anal. PDE},
  FJOURNAL = {Analysis \& PDE},
    VOLUME = {16},
      YEAR = {2023},
    NUMBER = {10},
     PAGES = {2267--2325},
      ISSN = {2157-5045,1948-206X},
   MRCLASS = {58J50 (35J05 35P05 35R02)},
  MRNUMBER = {4678142},
MRREVIEWER = {He-Jun\ Sun},
       DOI = {10.2140/apde.2023.16.2267},
       URL = {https://doi.org/10.2140/apde.2023.16.2267},
}

@article {RenZhang2024Improved,
    AUTHOR = {Ren, Tianyi and Zhang, An},
     TITLE = {Improved spectral cluster bounds for orthonormal systems},
   JOURNAL = {Forum Math.},
  FJOURNAL = {Forum Mathematicum},
    VOLUME = {36},
      YEAR = {2024},
    NUMBER = {5},
     PAGES = {1383--1392},
      ISSN = {0933-7741,1435-5337},
   MRCLASS = {58J50 (35J10 35P15 35R01)},
  MRNUMBER = {4790804},
MRREVIEWER = {Michael\ A.\ Perelmuter},
       DOI = {10.1515/forum-2023-0254},
       URL = {https://doi.org/10.1515/forum-2023-0254},
}

@article {ShaoYao2014UniformSobolev,
    AUTHOR = {Shao, Peng and Yao, Xiaohua},
     TITLE = {Uniform {S}obolev resolvent estimates for the
              {L}aplace-{B}eltrami operator on compact manifolds},
   JOURNAL = {Int. Math. Res. Not. IMRN},
  FJOURNAL = {International Mathematics Research Notices. IMRN},
      YEAR = {2014},
    NUMBER = {12},
     PAGES = {3439--3463},
      ISSN = {1073-7928,1687-0247},
   MRCLASS = {58J05 (35P15 35R01)},
  MRNUMBER = {3217667},
MRREVIEWER = {Yongqiang\ Fu},
       DOI = {10.1093/imrn/rnt051},
       URL = {https://doi.org/10.1093/imrn/rnt051},
}

@article{DonohoStark1989,
  author = {Donoho, David L. and Stark, Philip B.},
  title = {Uncertainty Principles and Signal Recovery},
  journal = {SIAM Journal on Applied Mathematics},
  volume = {49},
  number = {3},
  pages = {906--931},
  year = {1989}
}

@article{IosevichMayeli2025ACHA,
  author = {Iosevich, Alex and Mayeli, Azita},
  title = {Uncertainty Principles, Restriction, {B}ourgain's {$\Lambda_q$} Theorem, and Signal Recovery},
  journal = {Applied and Computational Harmonic Analysis},
  volume = {76},
  pages = {101734},
  year = {2025}
}

@book {Grieser1992PhD-Thesis,
    AUTHOR = {Grieser, Daniel},
     TITLE = {${L}^p$ bounds for eigenfunctions and spectral projections of
              the {L}aplacian near concave boundaries},
      NOTE = {Thesis (Ph.D.)--University of California, Los Angeles},
 PUBLISHER = {ProQuest LLC, Ann Arbor, MI},
      YEAR = {1992},
     PAGES = {87},
   MRCLASS = {99-05},
  MRNUMBER = {2688254},
       URL =
              {http://gateway.proquest.com/openurl?url_ver=Z39.88-2004&rft_val_fmt=info:ofi/fmt:kev:mtx:dissertation&res_dat=xri:pqdiss&rft_dat=xri:pqdiss:9301953},
}

@misc{IosevichPark2026UP-singular-potentials,
Author = {Alex Iosevich and Chamsol Park},
Title = {Uncertainty principles and singular potentials},
Year = {2026},
Howpublished = {\url{https://arxiv.org/abs/2604.15442}},
Eprint = {arXiv:2604.15442},
}

@article {Lindenstrauss2006Ergodicity,
    AUTHOR = {Lindenstrauss, Elon},
     TITLE = {Invariant measures and arithmetic quantum unique ergodicity},
   JOURNAL = {Ann. of Math. (2)},
  FJOURNAL = {Annals of Mathematics. Second Series},
    VOLUME = {163},
      YEAR = {2006},
    NUMBER = {1},
     PAGES = {165--219},
      ISSN = {0003-486X,1939-8980},
   MRCLASS = {11F72 (37A45 37D40)},
  MRNUMBER = {2195133},
MRREVIEWER = {Ze\'ev\ Rudnick},
       DOI = {10.4007/annals.2006.163.165},
       URL = {https://doi.org/10.4007/annals.2006.163.165},
}

@article {Soundararajan2010Ergodicity,
    AUTHOR = {Soundararajan, Kannan},
     TITLE = {Quantum unique ergodicity for {${\rm SL}_2(\mathbb{Z}
              )\backslash\mathbb{H}$}},
   JOURNAL = {Ann. of Math. (2)},
  FJOURNAL = {Annals of Mathematics. Second Series},
    VOLUME = {172},
      YEAR = {2010},
    NUMBER = {2},
     PAGES = {1529--1538},
      ISSN = {0003-486X,1939-8980},
   MRCLASS = {11F72 (58J51)},
  MRNUMBER = {2680500},
MRREVIEWER = {Gergely\ Harcos},
       DOI = {10.4007/annals.2010.172.1529},
       URL = {https://doi.org/10.4007/annals.2010.172.1529},
}

@article {DyatlovJinNonnenmacher2022Control,
    AUTHOR = {Dyatlov, Semyon and Jin, Long and Nonnenmacher, St\'ephane},
     TITLE = {Control of eigenfunctions on surfaces of variable curvature},
   JOURNAL = {J. Amer. Math. Soc.},
  FJOURNAL = {Journal of the American Mathematical Society},
    VOLUME = {35},
      YEAR = {2022},
    NUMBER = {2},
     PAGES = {361--465},
      ISSN = {0894-0347,1088-6834},
   MRCLASS = {58J51},
  MRNUMBER = {4374954},
MRREVIEWER = {Emmanuel\ Schenck},
       DOI = {10.1090/jams/979},
       URL = {https://doi.org/10.1090/jams/979},
}

\end{document}